\documentclass[12pt]{article}

\usepackage{amsmath,amsthm,amssymb,amsfonts}
\usepackage{mathrsfs}
\usepackage{enumitem}
\usepackage[margin=2.5cm]{geometry}
\usepackage{graphicx}
\usepackage{tikz}
\usetikzlibrary{arrows.meta,positioning}
\usepackage{hyperref}
\hypersetup{colorlinks=true,linkcolor=blue,citecolor=blue,urlcolor=blue}

\theoremstyle{plain}
\newtheorem{theorem}{Theorem}[section]

\newtheorem{proposition}[theorem]{Proposition}
\newtheorem{corollary}[theorem]{Corollary}

\theoremstyle{definition}
\newtheorem{definition}[theorem]{Definition}
\newtheorem{example}[theorem]{Example}

\theoremstyle{remark}
\newtheorem{remark}[theorem]{Remark}

\newcommand{\aura}{\mathfrak{a}}
\newcommand{\cla}{\operatorname{cl}_{\aura}}
\newcommand{\inta}{\operatorname{int}_{\aura}}
\newcommand{\cl}{\operatorname{cl}}
\newcommand{\inte}{\operatorname{int}}
\newcommand{\taua}{\tau_{\aura}}
\newcommand{\R}{\mathbb{R}}

\newcommand{\powerset}{\mathcal{P}}
\newcommand{\apr}{\underline{\operatorname{apr}}_{\aura}}
\newcommand{\Apr}{\overline{\operatorname{apr}}_{\aura}}

\title{\textbf{Aura Topological Spaces and Generalized Open Sets\\ with Applications to Rough Sets, Sensor Networks,\\ and Epidemic Modelling}}

\author{Ahu A\c{c}{\i}kg\"{o}z\\[6pt]
\small Department of Mathematics, Balikesir University,\\
\small Cagis Campus, 10145, Balikesir, Turkey\\
\small \texttt{ahuacikgoz@balikesir.edu.tr}}

\date{}

\begin{document}

\maketitle

\begin{abstract}
We equip a topological space $(X,\tau)$ with a function $\aura: X \to \tau$ satisfying the single axiom $x \in \aura(x)$. The resulting triple $(X, \tau, \aura)$, which we call an \emph{aura topological space}, provides a point-to-open-set assignment that differs from all existing auxiliary structures in topology---ideals, filters, grills, primals, and the various non-classical frameworks based on fuzzy, soft, or neutrosophic sets. The aura-closure operator $\cla(A) = \{x \in X : \aura(x) \cap A \neq \emptyset\}$ turns out to be an additive \v{C}ech closure operator; it satisfies extensivity, monotonicity, and finite additivity, but idempotency fails in general. Iterating $\cla$ transfinitely yields a Kuratowski closure whose topology $\taua^\infty$ satisfies $\taua^\infty \subseteq \taua \subseteq \tau$, where $\taua$ is the collection of all $\aura$-open sets. We introduce $\aura$-semi-open, $\aura$-pre-open, $\aura$-$\alpha$-open, and $\aura$-$\beta$-open sets, determine the complete hierarchy among these classes and their classical counterparts, and separate all non-coinciding classes by counterexamples on finite spaces as well as on the real line. The corresponding continuity notions and their decompositions are studied. Separation axioms $\aura$-$T_i$ ($i=0,1,2$) are introduced and their dependence on the choice of the scope function is demonstrated. Three applications are developed: (i)~upper and lower approximation operators that generalize Pawlak's rough set model without any equivalence relation, applied to a medical decision-making problem; (ii)~a wireless sensor network coverage model in which full coverage of a target region is characterised by $\aura$-openness; (iii)~an epidemic spread model in which the iterative spread operator tracks multi-step transmission chains and standard interventions (quarantine, social distancing) correspond to modifications of the scope function.
\end{abstract}


\medskip


\section{Introduction}

Over the past several decades, a major trend in general topology has been to enrich topological spaces with additional set-theoretic or algebraic data. The ideal topological spaces $(X, \tau, \mathcal{I})$, whose roots go back to Kuratowski \cite{Kuratowski1933} and Vaidyanathaswamy \cite{Vaidyanathaswamy1945}, were brought to the mainstream by Jankovi\'{c} and Hamlett \cite{Jankovic1990}. In this framework an ideal $\mathcal{I}$ (a hereditary, finitely additive family containing $\emptyset$) interacts with the topology through the local function $A^*(\mathcal{I}, \tau)$ and produces a finer topology $\tau^*$. Dual notions followed: Cartan's filters \cite{Cartan1937}, Choquet's grills \cite{Choquet1947} revived by Roy and Mukherjee \cite{Roy2007}, and the recent primals of Acharjee, \"{O}zko\c{c}, and Issaka \cite{Acharjee2022}. The primal concept, in particular, has sparked a wave of activity on operators \cite{AlOmari2023,Alghamdi2024}, separation axioms \cite{AlSaadi2024}, and compatibility conditions \cite{Alqahtani2024}.

Alongside these developments, topology has been blended with fuzzy sets (Zadeh \cite{Zadeh1965}, Chang \cite{Chang1968}), soft sets (Molodtsov \cite{Molodtsov1999}, Shabir and Naz \cite{Shabir2011}), neutrosophic sets (Smarandache \cite{Smarandache1998}, Salama and Alblowi \cite{Salama2012}), and rough sets (Pawlak \cite{Pawlak1982}). If one steps back and looks at the picture as a whole, the existing auxiliary structures can be grouped as follows: (i) subcollections of $\powerset(X)$ (ideals, filters, grills, primals); (ii) second topologies as in Kelly's bitopological spaces \cite{Kelly1963}; (iii) membership-grade functions (fuzzy, intuitionistic fuzzy, neutrosophic); (iv) parametric families (soft sets); (v) equivalence relations (Pawlak rough sets).

The starting point of the present work is the observation that none of these structures captures the idea of assigning to each point a \emph{single fixed open neighborhood}. We propose exactly this: a function $\aura: X \to \tau$ satisfying $x \in \aura(x)$ for all $x \in X$. We call $\aura$ a \emph{scope function} and the triple $(X, \tau, \aura)$ an \textbf{aura topological space}. The name is meant to suggest that each point $x$ carries a ``scope of influence'' $\aura(x)$ around it. We stress that $\aura$ is not a subcollection of $\powerset(X)$, not a second topology, not a membership function, and not a relation; it is a \emph{point-to-open-set assignment}.

We should note that neighborhood assignments $\phi: X \to \tau$ with $x \in \phi(x)$ already appear in the theory of D-spaces (van Douwen and Pfeffer \cite{vanDouwen1979}). There, however, such assignments serve as universally quantified variables---one says ``for every neighborhood assignment $\phi$...''---rather than as a fixed ingredient of the space. Our approach goes in the opposite direction: we fix $\aura$ once and for all and then derive operators, open-set classes, continuity notions, and separation axioms from it.

Part of our motivation comes from outside pure mathematics. In many real-world situations each point of a space carries a natural ``scope'': a wireless sensor has a detection range, a base station has a coverage area, an infected individual can transmit a disease within a certain radius, a pixel in a digital image interacts with its neighbours. All of these are instances of a fixed open-neighbourhood assignment, i.e.\ of a scope function $\aura$. The iterative aura-closure $\cla^n$ then models multi-step propagation (signal relay, epidemic spread, region growing in image segmentation), while the transitivity condition on $\aura$ captures situations where ``what my scope covers is itself fully within my scope.'' To our knowledge, no existing topological framework handles these scenarios in a single, unified structure. Sections~7.2 and~7.3 below give concrete models for wireless coverage and disease transmission that illustrate this point.

The paper is structured as follows. Section~2 collects the preliminary material. Section~3 introduces aura topological spaces and establishes the basic properties of $\cla$ and $\inta$; in particular, $\cla$ is shown to be an additive \v{C}ech closure operator whose iteration $\cla^\infty$ is a Kuratowski closure, and the topology $\taua$ of $\aura$-open sets is proved to satisfy $\taua^\infty \subseteq \taua \subseteq \tau$. In Section~4 we define five generalized open-set classes and construct a complete implication diagram, using counterexamples on both finite spaces and the real line. Section~5 treats $\aura$-continuity and its decompositions. Section~6 introduces $\aura$-$T_0$, $\aura$-$T_1$, and $\aura$-$T_2$ separation. Section~7 presents three applications: to rough set theory (with a medical decision-making example), to wireless sensor network coverage, and to epidemic spread modelling.

\section{Preliminaries}

Throughout this paper, $(X, \tau)$ denotes a topological space. For a subset $A$ of $X$, the closure of $A$ in $(X,\tau)$ is denoted by $\cl(A)$ and the interior by $\inte(A)$. The complement of $A$ is denoted by $A^c = X \setminus A$. The power set of $X$ is denoted by $\powerset(X)$. For $x \in X$, we denote by $\tau(x) = \{U \in \tau : x \in U\}$ the collection of all open neighborhoods of $x$.

\begin{definition}[\cite{Levine1963}]\label{def:semi-open}
A subset $A$ of a topological space $(X,\tau)$ is called \emph{semi-open} if $A \subseteq \cl(\inte(A))$. The complement of a semi-open set is called \emph{semi-closed}. The collection of all semi-open sets is denoted by $SO(X,\tau)$.
\end{definition}

\begin{definition}[\cite{Mashhour1982}]\label{def:pre-open}
A subset $A$ of a topological space $(X,\tau)$ is called \emph{pre-open} if $A \subseteq \inte(\cl(A))$. The complement of a pre-open set is called \emph{pre-closed}. The collection of all pre-open sets is denoted by $PO(X,\tau)$.
\end{definition}

\begin{definition}[\cite{Njastad1965}]\label{def:alpha-open}
A subset $A$ of a topological space $(X,\tau)$ is called \emph{$\alpha$-open} if $A \subseteq \inte(\cl(\inte(A)))$. The collection of all $\alpha$-open sets is denoted by $\alpha O(X,\tau)$.
\end{definition}

\begin{definition}[\cite{AbdElMonsef1983}]\label{def:beta-open}
A subset $A$ of a topological space $(X,\tau)$ is called \emph{$\beta$-open} (or \emph{semi-pre-open}) if $A \subseteq \cl(\inte(\cl(A)))$. The collection of all $\beta$-open sets is denoted by $\beta O(X,\tau)$.
\end{definition}

\begin{definition}[\cite{Cech1966}]
A function $c: \powerset(X) \to \powerset(X)$ is called a \emph{\v{C}ech closure operator} if it satisfies the following conditions for all $A, B \subseteq X$:
\begin{enumerate}[label=(C\arabic*)]
    \item $c(\emptyset) = \emptyset$;
    \item $A \subseteq c(A)$ (extensivity);
    \item $A \subseteq B \implies c(A) \subseteq c(B)$ (monotonicity).
\end{enumerate}
If $c$ additionally satisfies:
\begin{enumerate}[label=(C\arabic*)]
\setcounter{enumi}{3}
    \item $c(A \cup B) = c(A) \cup c(B)$ (finite additivity),
\end{enumerate}
then $c$ is called an \emph{additive \v{C}ech closure operator}. If $c$ further satisfies:
\begin{enumerate}[label=(C\arabic*)]
\setcounter{enumi}{4}
    \item $c(c(A)) = c(A)$ (idempotency),
\end{enumerate}
then $c$ is a \emph{Kuratowski closure operator} and generates a topology on $X$.
\end{definition}

\begin{definition}[\cite{Pawlak1982}]\label{def:rough}
Let $R$ be an equivalence relation on a non-empty set $U$. For each $x \in U$, the equivalence class of $x$ is $[x]_R = \{y \in U : (x,y) \in R\}$. The pair $(U, R)$ is called a \emph{Pawlak approximation space}. For any $A \subseteq U$:
\begin{itemize}
    \item The \emph{lower approximation} is $\underline{\operatorname{apr}}_R(A) = \{x \in U : [x]_R \subseteq A\}$.
    \item The \emph{upper approximation} is $\overline{\operatorname{apr}}_R(A) = \{x \in U : [x]_R \cap A \neq \emptyset\}$.
    \item The \emph{boundary region} is $\operatorname{bnd}_R(A) = \overline{\operatorname{apr}}_R(A) \setminus \underline{\operatorname{apr}}_R(A)$.
\end{itemize}
\end{definition}

\section{Aura Topological Spaces}

In this section, we introduce the central concept of this paper and investigate the fundamental properties of the associated operators.

\begin{definition}\label{def:aura}
Let $(X, \tau)$ be a topological space. A function $\aura: X \to \tau$ is called a \textbf{scope function} (or an \textbf{aura function}) on $(X, \tau)$ if it satisfies:
\begin{equation}\label{eq:aura-axiom}
    x \in \aura(x) \quad \text{for every } x \in X.
\end{equation}
The triple $(X, \tau, \aura)$ is called an \textbf{aura topological space} (briefly, an $\aura$-\textbf{space}).
\end{definition}

\begin{remark}\label{rem:aura-selection}
The scope function $\aura$ selects, for each point $x$, a particular open neighborhood $\aura(x) \in \tau(x)$. One may think of $\aura(x)$ as the ``scope'' of $x$---the part of the space that $x$ can directly interact with. Different choices of $\aura$ on the same topological space $(X,\tau)$ yield different $\aura$-spaces with potentially different properties.
\end{remark}

\begin{example}\label{ex:finite-aura}
Let $X = \{a, b, c, d\}$ and $\tau = \{\emptyset, \{a\}, \{b\}, \{a,b\}, \{a,b,c\}, X\}$.
Define $\aura: X \to \tau$ by
\[
\aura(a) = \{a\}, \quad \aura(b) = \{a,b\}, \quad \aura(c) = \{a,b,c\}, \quad \aura(d) = X.
\]
Then $(X, \tau, \aura)$ is an $\aura$-space since $x \in \aura(x)$ for each $x \in X$.
\end{example}

\begin{example}\label{ex:real-aura}
Let $(\R, \tau_u)$ be the real line with the usual topology. For a fixed $\varepsilon > 0$, define $\aura_\varepsilon: \R \to \tau_u$ by
\[
\aura_\varepsilon(x) = (x - \varepsilon, x + \varepsilon)
\]
for every $x \in \R$. Then $(\R, \tau_u, \aura_\varepsilon)$ is an $\aura$-space. This can be interpreted as assigning each point a fixed ``resolution window'' of radius $\varepsilon$.
\end{example}

\begin{example}\label{ex:variable-aura}
Let $(\R, \tau_u)$ be as above. Define $\aura: \R \to \tau_u$ by
\[
\aura(x) = \left(x - \frac{1}{1+x^2}, \; x + \frac{1}{1+x^2}\right)
\]
for every $x \in \R$. Then $(\R, \tau_u, \aura)$ is an $\aura$-space where the scope narrows as $|x| \to \infty$, modeling a variable-resolution observation.
\end{example}

\subsection{The Aura-Closure Operator}

\begin{definition}\label{def:aura-closure}
Let $(X, \tau, \aura)$ be an $\aura$-space. For $A \subseteq X$, the \textbf{aura-closure} of $A$ is defined by
\begin{equation}\label{eq:aura-closure}
    \cla(A) = \{x \in X : \aura(x) \cap A \neq \emptyset\}.
\end{equation}
\end{definition}

\begin{remark}
The aura-closure of $A$ consists of all points whose aura intersects $A$. In other words, $x \in \cla(A)$ if and only if $A$ meets $\aura(x)$.
\end{remark}

\begin{theorem}\label{thm:cech}
Let $(X, \tau, \aura)$ be an $\aura$-space. The operator $\cla: \powerset(X) \to \powerset(X)$ satisfies the following properties for all $A, B \subseteq X$:
\begin{enumerate}[label=(\alph*)]
    \item $\cla(\emptyset) = \emptyset$.
    \item $A \subseteq \cla(A)$.
    \item $A \subseteq B \implies \cla(A) \subseteq \cla(B)$.
    \item $\cla(A \cup B) = \cla(A) \cup \cla(B)$.
    \item $\cl(A) \subseteq \cla(A)$.
\end{enumerate}
Hence, $\cla$ is an additive \v{C}ech closure operator satisfying the first four Kuratowski axioms.
\end{theorem}

\begin{proof}
\begin{enumerate}[label=(\alph*)]
    \item Since $\aura(x) \cap \emptyset = \emptyset$ for every $x \in X$, we have $\cla(\emptyset) = \emptyset$.

    \item Let $x \in A$. By (\ref{eq:aura-axiom}), $x \in \aura(x)$, so $x \in \aura(x) \cap A$. Thus $\aura(x) \cap A \neq \emptyset$, which gives $x \in \cla(A)$.

    \item Let $A \subseteq B$ and $x \in \cla(A)$. Then $\aura(x) \cap A \neq \emptyset$, so there exists $y \in \aura(x) \cap A \subseteq \aura(x) \cap B$. Hence $\aura(x) \cap B \neq \emptyset$ and $x \in \cla(B)$.

    \item For any $x \in X$:
    \begin{align*}
        x \in \cla(A \cup B) &\iff \aura(x) \cap (A \cup B) \neq \emptyset \\
        &\iff (\aura(x) \cap A) \cup (\aura(x) \cap B) \neq \emptyset \\
        &\iff \aura(x) \cap A \neq \emptyset \text{ or } \aura(x) \cap B \neq \emptyset \\
        &\iff x \in \cla(A) \text{ or } x \in \cla(B) \\
        &\iff x \in \cla(A) \cup \cla(B).
    \end{align*}

    \item Let $x \in \cl(A)$. Then every open neighborhood of $x$ intersects $A$. Since $\aura(x)$ is an open neighborhood of $x$, we have $\aura(x) \cap A \neq \emptyset$, so $x \in \cla(A)$. \qedhere
\end{enumerate}
\end{proof}

\begin{theorem}\label{thm:not-idempotent}
The operator $\cla$ is not idempotent in general; that is, there exists an $\aura$-space $(X, \tau, \aura)$ and a subset $A \subseteq X$ such that $\cla(\cla(A)) \neq \cla(A)$.
\end{theorem}

\begin{proof}
Let $X = \{a, b, c\}$, $\tau = \powerset(X)$ (discrete topology), and define
\[
\aura(a) = \{a, b\}, \quad \aura(b) = \{b, c\}, \quad \aura(c) = \{c\}.
\]
Let $A = \{c\}$. Then:
\[
\cla(A) = \{x \in X : \aura(x) \cap \{c\} \neq \emptyset\} = \{b, c\},
\]
since $\aura(b) \cap \{c\} = \{c\} \neq \emptyset$ and $\aura(c) \cap \{c\} = \{c\} \neq \emptyset$, but $\aura(a) \cap \{c\} = \emptyset$. Now:
\[
\cla(\cla(A)) = \cla(\{b, c\}) = \{x \in X : \aura(x) \cap \{b, c\} \neq \emptyset\} = \{a, b, c\} = X,
\]
since $\aura(a) \cap \{b, c\} = \{b\} \neq \emptyset$. Therefore $\cla(A) = \{b,c\} \neq X = \cla(\cla(A))$.
\end{proof}

\begin{theorem}\label{thm:cl-vs-cla}
Let $(X, \tau, \aura)$ be an $\aura$-space and $A \subseteq X$. Then:
\begin{enumerate}[label=(\alph*)]
    \item $\cl(A) \subseteq \cla(A)$.
    \item The inclusion can be strict.
    \item $\cla(A) \subseteq \bigcup_{x \in A} \aura(x)$ need not hold in general.
\end{enumerate}
\end{theorem}

\begin{proof}
Part (a) is proven in Theorem \ref{thm:cech}(e). For part (b), consider $X = \{a,b,c\}$, $\tau = \{\emptyset, \{a\}, \{b,c\}, X\}$, and $\aura(a) = \{a\}$, $\aura(b) = \{b,c\}$, $\aura(c) = \{b,c\}$. Let $A = \{b\}$. Then $\cl(A) = \{b,c\}$ and $\cla(A) = \{b, c\}$ since $\aura(b) \cap \{b\} = \{b\} \neq \emptyset$ and $\aura(c) \cap \{b\} = \{b\} \neq \emptyset$, but $\aura(a) \cap \{b\} = \emptyset$. Here the two closures coincide. Instead, take $X = \{a,b,c\}$, $\tau = \powerset(X)$, $\aura(a) = \{a,b\}$, $\aura(b) = \{b\}$, $\aura(c) = \{c\}$, and $A = \{b\}$. Then $\cl(A) = \{b\}$ (discrete topology) but $\cla(A) = \{a, b\}$ since $\aura(a) \cap \{b\} = \{b\} \neq \emptyset$. The inclusion is strict.

For part (c), with the same space, $\bigcup_{x \in A} \aura(x) = \aura(b) = \{b\}$, but $\cla(A) = \{a,b\} \not\subseteq \{b\}$.
\end{proof}

\subsection{The Aura-Interior Operator}

\begin{definition}\label{def:aura-interior}
Let $(X, \tau, \aura)$ be an $\aura$-space. For $A \subseteq X$, the \textbf{aura-interior} of $A$ is defined by
\begin{equation}\label{eq:aura-interior}
    \inta(A) = \{x \in A : \aura(x) \subseteq A\}.
\end{equation}
\end{definition}

\begin{theorem}\label{thm:int-properties}
Let $(X, \tau, \aura)$ be an $\aura$-space. The operator $\inta: \powerset(X) \to \powerset(X)$ satisfies the following for all $A, B \subseteq X$:
\begin{enumerate}[label=(\alph*)]
    \item $\inta(X) = X$.
    \item $\inta(A) \subseteq A$.
    \item $A \subseteq B \implies \inta(A) \subseteq \inta(B)$.
    \item $\inta(A \cap B) = \inta(A) \cap \inta(B)$.
    \item $\inta(A) \subseteq \inte(A)$.
    \item $\inta(A) = A \setminus \cla(A^c)$.
\end{enumerate}
\end{theorem}

\begin{proof}
\begin{enumerate}[label=(\alph*)]
    \item For every $x \in X$, $\aura(x) \subseteq X$, so $x \in \inta(X)$.

    \item By definition, $\inta(A) = \{x \in A : \aura(x) \subseteq A\} \subseteq A$.

    \item Let $x \in \inta(A)$. Then $x \in A$ and $\aura(x) \subseteq A \subseteq B$. Since $x \in A \subseteq B$, we have $x \in B$ and $\aura(x) \subseteq B$, so $x \in \inta(B)$.

    \item Let $x \in \inta(A \cap B)$. Then $x \in A \cap B$ and $\aura(x) \subseteq A \cap B$. Hence $x \in A$, $\aura(x) \subseteq A$ and $x \in B$, $\aura(x) \subseteq B$. So $x \in \inta(A) \cap \inta(B)$. Conversely, if $x \in \inta(A) \cap \inta(B)$, then $x \in A \cap B$ and $\aura(x) \subseteq A \cap B$, giving $x \in \inta(A \cap B)$.

    \item Let $x \in \inta(A)$. Then $x \in A$ and $\aura(x) \subseteq A$. Since $\aura(x)$ is an open set containing $x$ and contained in $A$, we have $x \in \inte(A)$.

    \item We have:
    \begin{align*}
        x \in \inta(A) &\iff x \in A \text{ and } \aura(x) \subseteq A \\
        &\iff x \in A \text{ and } \aura(x) \cap A^c = \emptyset \\
        &\iff x \in A \text{ and } x \notin \cla(A^c) \\
        &\iff x \in A \setminus \cla(A^c). \qedhere
    \end{align*}
\end{enumerate}
\end{proof}

\subsection{The Aura Topology}

\begin{definition}\label{def:aura-open}
Let $(X, \tau, \aura)$ be an $\aura$-space. A subset $A$ of $X$ is called \textbf{$\aura$-open} if for every $x \in A$, $\aura(x) \subseteq A$. Equivalently, $A$ is $\aura$-open if and only if $\inta(A) = A$. The collection of all $\aura$-open sets is denoted by $\taua$.
\end{definition}

\begin{theorem}\label{thm:aura-topology}
Let $(X, \tau, \aura)$ be an $\aura$-space. Then:
\begin{enumerate}[label=(\alph*)]
    \item $\taua$ is a topology on $X$.
    \item $\taua \subseteq \tau$.
\end{enumerate}
\end{theorem}

\begin{proof}
\begin{enumerate}[label=(\alph*)]
    \item We verify the topology axioms:
    \begin{itemize}
        \item $\emptyset \in \taua$: vacuously true, since there is no $x \in \emptyset$.
        \item $X \in \taua$: for every $x \in X$, $\aura(x) \subseteq X$.
        \item \emph{Finite intersection}: Let $A, B \in \taua$ and $x \in A \cap B$. Then $\aura(x) \subseteq A$ and $\aura(x) \subseteq B$, so $\aura(x) \subseteq A \cap B$.
        \item \emph{Arbitrary union}: Let $\{A_i\}_{i \in I} \subseteq \taua$ and $x \in \bigcup_{i \in I} A_i$. Then $x \in A_j$ for some $j \in I$, so $\aura(x) \subseteq A_j \subseteq \bigcup_{i \in I} A_i$.
    \end{itemize}

    \item Let $A \in \taua$. For every $x \in A$, $\aura(x) \subseteq A$ and $\aura(x) \in \tau$ with $x \in \aura(x)$. Thus $A = \bigcup_{x \in A} \aura(x) \in \tau$, since arbitrary unions of open sets are open. Alternatively, by Theorem \ref{thm:int-properties}(e), $A = \inta(A) \subseteq \inte(A) \subseteq A$, so $A = \inte(A) \in \tau$. \qedhere
\end{enumerate}
\end{proof}

\begin{example}\label{ex:aura-topology-finite}
Consider the $\aura$-space from Example \ref{ex:finite-aura}. The $\aura$-open sets are:
\begin{itemize}
    \item $\emptyset$: trivially $\aura$-open.
    \item $\{a\}$: $\aura(a) = \{a\} \subseteq \{a\}$. Yes.
    \item $\{a,b\}$: $\aura(a) = \{a\} \subseteq \{a,b\}$ and $\aura(b) = \{a,b\} \subseteq \{a,b\}$. Yes.
    \item $\{a,b,c\}$: $\aura(a) = \{a\} \subseteq \{a,b,c\}$, $\aura(b) = \{a,b\} \subseteq \{a,b,c\}$, and $\aura(c) = \{a,b,c\} \subseteq \{a,b,c\}$. Yes.
    \item $X$: Yes.
    \item $\{b\}$: $\aura(b) = \{a,b\} \not\subseteq \{b\}$. No.
    \item $\{c\}$: $\aura(c) = \{a,b,c\} \not\subseteq \{c\}$. No.
    \item $\{d\}$: $\aura(d) = X \not\subseteq \{d\}$. No.
\end{itemize}
Hence $\taua = \{\emptyset, \{a\}, \{a,b\}, \{a,b,c\}, X\}$. Note that $\{b\} \in \tau$ but $\{b\} \notin \taua$, confirming that $\taua \subsetneq \tau$ in general.
\end{example}

\begin{proposition}\label{prop:aura-base}
Let $(X, \tau, \aura)$ be an $\aura$-space. The family $\mathcal{B}_\aura = \{\aura(x) : x \in X\}$ is a subbase for $\taua$. More precisely, if $A \in \taua$, then $A = \bigcup_{x \in A} \aura(x)$, so $\mathcal{B}_\aura$ is a cover-base for $\taua$.
\end{proposition}

\begin{proof}
Let $A \in \taua$. For each $x \in A$, $x \in \aura(x) \subseteq A$, so $A \subseteq \bigcup_{x \in A} \aura(x) \subseteq A$. Conversely, any union of elements of $\mathcal{B}_\aura$ is $\aura$-open: if $x \in \bigcup_{i \in I} \aura(x_i)$, then $x \in \aura(x_j)$ for some $j$, and we need $\aura(x) \subseteq \bigcup_{i \in I} \aura(x_i)$. This holds if and only if for every $y \in \aura(x)$, $y \in \aura(x_k)$ for some $k$. This need not be automatic, so $\mathcal{B}_\aura$ serves as a cover-base rather than a base in general.
\end{proof}

\begin{definition}\label{def:transitive}
An $\aura$-space $(X, \tau, \aura)$ is called \textbf{transitive} if for every $x \in X$ and every $y \in \aura(x)$, $\aura(y) \subseteq \aura(x)$.
\end{definition}

\begin{proposition}\label{prop:transitive-base}
If $(X, \tau, \aura)$ is a transitive $\aura$-space, then $\mathcal{B}_\aura = \{\aura(x) : x \in X\}$ is a base for $\taua$, and the aura-closure $\cla$ is idempotent.
\end{proposition}

\begin{proof}
Let $y \in \aura(x)$. By transitivity, $\aura(y) \subseteq \aura(x)$, so $\aura(x)$ is $\aura$-open, i.e., $\aura(x) \in \taua$ for every $x \in X$. Since every $\aura$-open set is a union of such sets, $\mathcal{B}_\aura$ is a base for $\taua$.

For idempotency, let $x \in \cla(\cla(A))$. Then there exists $y \in \aura(x) \cap \cla(A)$, meaning there exists $z \in \aura(y) \cap A$. By transitivity, $\aura(y) \subseteq \aura(x)$, so $z \in \aura(x) \cap A$, giving $x \in \cla(A)$.
\end{proof}

\subsection{Iterative Aura-Closure}

Since $\cla$ is not idempotent in general, we define the iterative closures.

\begin{definition}\label{def:iterative}
Let $(X, \tau, \aura)$ be an $\aura$-space and $A \subseteq X$. Define:
\begin{align}
    \cla^0(A) &= A, \\
    \cla^{n+1}(A) &= \cla(\cla^n(A)) \quad \text{for } n \geq 0, \\
    \cla^\infty(A) &= \bigcup_{n=0}^{\infty} \cla^n(A).
\end{align}
\end{definition}

\begin{theorem}\label{thm:chain}
Let $(X, \tau, \aura)$ be an $\aura$-space and $A \subseteq X$. Then:
\begin{enumerate}[label=(\alph*)]
    \item $A \subseteq \cla(A) \subseteq \cla^2(A) \subseteq \cdots \subseteq \cla^\infty(A)$.
    \item $\cla^\infty$ is a Kuratowski closure operator.
    \item The topology generated by $\cla^\infty$ satisfies $\tau_\aura^\infty \subseteq \taua \subseteq \tau$.
\end{enumerate}
\end{theorem}

\begin{proof}
\begin{enumerate}[label=(\alph*)]
    \item By extensivity of $\cla$, $\cla^n(A) \subseteq \cla^{n+1}(A)$ for every $n \geq 0$.

    \item We verify the Kuratowski axioms for $\cla^\infty$:
    \begin{itemize}
        \item $\cla^\infty(\emptyset) = \bigcup_{n=0}^\infty \cla^n(\emptyset) = \emptyset$.
        \item $A \subseteq \cla^\infty(A)$: since $A = \cla^0(A)$.
        \item \emph{Additivity}: $\cla^\infty(A \cup B) = \bigcup_n \cla^n(A \cup B) = \bigcup_n (\cla^n(A) \cup \cla^n(B)) = \cla^\infty(A) \cup \cla^\infty(B)$.
        \item \emph{Idempotency}: Let $x \in \cla^\infty(\cla^\infty(A))$. Then $x \in \cla^m(\cla^\infty(A))$ for some $m$. We show by induction on $m$ that $x \in \cla^\infty(A)$. For $m = 0$, $x \in \cla^\infty(A)$. For $m = 1$, $x \in \cla(\cla^\infty(A))$ means $\aura(x) \cap \cla^\infty(A) \neq \emptyset$, so there exist $y \in \aura(x)$ and $k$ with $y \in \cla^k(A)$. Then $x \in \cla(\cla^k(A)) = \cla^{k+1}(A) \subseteq \cla^\infty(A)$. The general case follows by induction.
    \end{itemize}

    \item If $A \in \taua$, then for every $x \in A$, $\aura(x) \subseteq A$. This implies $\cla(A^c) \subseteq A^c$ (if $x \notin A$, or if $x \in A$ then $\aura(x) \subseteq A$ so $\aura(x) \cap A^c = \emptyset$). By induction, $\cla^n(A^c) \subseteq A^c$ for all $n$, so $\cla^\infty(A^c) = A^c$, meaning $A^c$ is closed in $\tau_\aura^\infty$ and $A$ is open in $\tau_\aura^\infty$. Wait---we need to be more careful. Actually, $A \in \taua$ means $\inta(A) = A$, i.e., $\cla(A^c) \cap A = \emptyset$, i.e., $\cla(A^c) \subseteq A^c$. Then $\cla^2(A^c) = \cla(\cla(A^c)) \subseteq \cla(A^c) \subseteq A^c$ by monotonicity. By induction, $\cla^n(A^c) \subseteq A^c$ for all $n$, so $\cla^\infty(A^c) \subseteq A^c$. Since $A^c \subseteq \cla^\infty(A^c)$, we get $\cla^\infty(A^c) = A^c$, confirming $A \in \tau_\aura^\infty$.

    The inclusion $\taua \subseteq \tau$ follows from Theorem \ref{thm:aura-topology}(b). \qedhere
\end{enumerate}
\end{proof}

\begin{remark}\label{rem:finite-stabilization}
If $X$ is finite, then the chain $\cla^0(A) \subseteq \cla^1(A) \subseteq \cla^2(A) \subseteq \cdots$ must stabilize in at most $|X|$ steps, so $\cla^\infty(A) = \cla^{|X|}(A)$.
\end{remark}

\subsection{Special Types of Aura Functions}

\begin{definition}\label{def:special-auras}
Let $(X, \tau, \aura)$ be an $\aura$-space. We say that $\aura$ is:
\begin{enumerate}[label=(\alph*)]
    \item \emph{trivial} if $\aura(x) = X$ for every $x \in X$;
    \item \emph{discrete} if $\aura(x) = \{x\}$ for every $x \in X$ (possible only if $\tau$ is the discrete topology);
    \item \emph{monotone} if $x \in \aura(y)$ implies $\aura(x) \subseteq \aura(y)$ (i.e., transitive);
    \item \emph{symmetric} if $y \in \aura(x)$ implies $x \in \aura(y)$.
\end{enumerate}
\end{definition}

\begin{proposition}\label{prop:special-cases}
Let $(X, \tau, \aura)$ be an $\aura$-space.
\begin{enumerate}[label=(\alph*)]
    \item If $\aura$ is trivial, then $\cla(A) = X$ for every nonempty $A$, and $\taua = \{\emptyset, X\}$.
    \item If $\aura$ is discrete, then $\cla(A) = A$ for every $A$, and $\taua = \powerset(X)$.
    \item If $\aura$ is symmetric, then $x \in \cla(\{y\})$ if and only if $y \in \cla(\{x\})$.
\end{enumerate}
\end{proposition}

\begin{proof}
\begin{enumerate}[label=(\alph*)]
    \item If $A \neq \emptyset$, pick $a \in A$. For any $x \in X$, $\aura(x) = X \ni a$, so $\aura(x) \cap A \neq \emptyset$ and $x \in \cla(A)$.

    \item $\aura(x) \cap A = \{x\} \cap A \neq \emptyset$ iff $x \in A$, so $\cla(A) = A$.

    \item $x \in \cla(\{y\})$ iff $y \in \aura(x)$ iff (by symmetry) $x \in \aura(y)$ iff $y \in \cla(\{x\})$. \qedhere
\end{enumerate}
\end{proof}

\section{Generalized Open Sets in Aura Spaces}

In this section, we introduce five new classes of generalized open sets by combining the aura-closure operator $\cla$ with the classical interior and closure operators.

\begin{definition}\label{def:generalized-open}
Let $(X, \tau, \aura)$ be an $\aura$-space. A subset $A$ of $X$ is called:
\begin{enumerate}[label=(\alph*)]
    \item \textbf{$\aura$-semi-open} if $A \subseteq \cla(\inte(A))$;
    \item \textbf{$\aura$-pre-open} if $A \subseteq \inte(\cla(A))$;
    \item \textbf{$\aura$-$\alpha$-open} if $A \subseteq \inte(\cla(\inte(A)))$;
    \item \textbf{$\aura$-$\beta$-open} if $A \subseteq \cla(\inte(\cla(A)))$;
    \item \textbf{$\aura$-$b$-open} if $A \subseteq \cla(\inte(A)) \cup \inte(\cla(A))$.
\end{enumerate}
The collection of all $\aura$-semi-open (resp.\ $\aura$-pre-open, $\aura$-$\alpha$-open, $\aura$-$\beta$-open, $\aura$-$b$-open) sets is denoted by $\aura SO(X)$ (resp.\ $\aura PO(X)$, $\aura\alpha O(X)$, $\aura\beta O(X)$, $\aura bO(X)$).
\end{definition}

\begin{theorem}\label{thm:hierarchy}
Let $(X, \tau, \aura)$ be an $\aura$-space. The following inclusions hold:
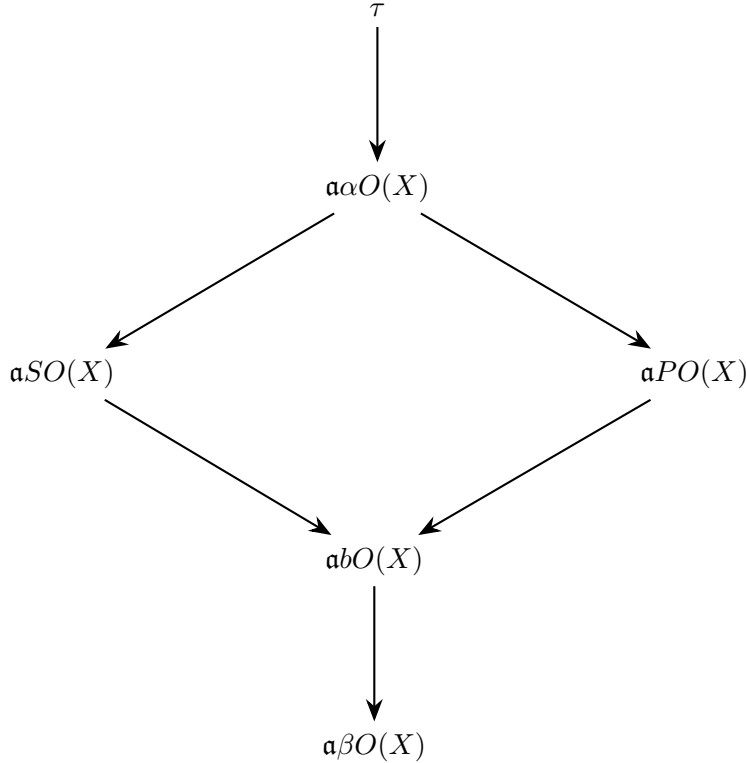
\begin{center}
\begin{tikzpicture}[
    node distance=1.8cm and 2.5cm,
    every node/.style={font=\small},
    arrow/.style={-{Stealth[length=3mm]}, thick}
]
\node (open) {$\tau$};
\node (aopen) [below=of open] {$\aura\alpha O(X)$};
\node (semi) [below left=of aopen] {$\aura SO(X)$};
\node (pre) [below right=of aopen] {$\aura PO(X)$};
\node (b) [below right=of semi] {$\aura bO(X)$};
\node (beta) [below=of b] {$\aura\beta O(X)$};

\draw[arrow] (open) -- (aopen);
\draw[arrow] (aopen) -- (semi);
\draw[arrow] (aopen) -- (pre);
\draw[arrow] (semi) -- (b);
\draw[arrow] (pre) -- (b);
\draw[arrow] (b) -- (beta);
\end{tikzpicture}
\end{center}
where $A \to B$ means every $A$-type set is also a $B$-type set. Moreover, each implication is strict in general.
\end{theorem}

\begin{proof}
\textbf{$\tau \subseteq \aura\alpha O(X)$}: Let $A \in \tau$. Then $\inte(A) = A$, so $\cla(\inte(A)) = \cla(A) \supseteq \cl(A) \supseteq A$. Hence $\inte(\cla(\inte(A))) = \inte(\cla(A)) \supseteq \inte(\cl(A)) \supseteq \inte(A) = A$ (since $A$ is open and $A \subseteq \cl(A)$ implies $A = \inte(A) \subseteq \inte(\cl(A))$). Thus $A \subseteq \inte(\cla(\inte(A)))$.

\textbf{$\aura\alpha O(X) \subseteq \aura SO(X)$}: If $A \subseteq \inte(\cla(\inte(A)))$, then since $\inte(B) \subseteq B$ for any $B$, we get $A \subseteq \inte(\cla(\inte(A))) \subseteq \cla(\inte(A))$.

\textbf{$\aura\alpha O(X) \subseteq \aura PO(X)$}: If $A \subseteq \inte(\cla(\inte(A)))$, then $\inte(A) \subseteq A$ gives $\cla(\inte(A)) \subseteq \cla(A)$ (monotonicity), so $A \subseteq \inte(\cla(\inte(A))) \subseteq \inte(\cla(A))$.

\textbf{$\aura SO(X) \subseteq \aura bO(X)$}: If $A \subseteq \cla(\inte(A))$, then $A \subseteq \cla(\inte(A)) \subseteq \cla(\inte(A)) \cup \inte(\cla(A))$.

\textbf{$\aura PO(X) \subseteq \aura bO(X)$}: Similarly.

\textbf{$\aura bO(X) \subseteq \aura\beta O(X)$}: If $A \subseteq \cla(\inte(A)) \cup \inte(\cla(A))$, note that $\inte(A) \subseteq \inte(\cla(A))$ (since $A \subseteq \cla(A)$). Hence $\cla(\inte(A)) \subseteq \cla(\inte(\cla(A)))$. Also $\inte(\cla(A)) \subseteq \cla(\inte(\cla(A)))$ (by extensivity of $\cla$). Therefore $A \subseteq \cla(\inte(\cla(A)))$.
\end{proof}

\begin{theorem}\label{thm:classical-comparison}
Let $(X, \tau, \aura)$ be an $\aura$-space. Since $\cl(A) \subseteq \cla(A)$ for all $A \subseteq X$, we have:
\begin{enumerate}[label=(\alph*)]
    \item $SO(X, \tau) \subseteq \aura SO(X)$;
    \item $PO(X, \tau) \subseteq \aura PO(X)$;
    \item $\alpha O(X, \tau) \subseteq \aura\alpha O(X)$;
    \item $\beta O(X, \tau) \subseteq \aura\beta O(X)$.
\end{enumerate}
Each inclusion can be strict.
\end{theorem}

\begin{proof}
Since $\cl(A) \subseteq \cla(A)$, we have $\cla(\inte(A)) \supseteq \cl(\inte(A))$. If $A \in SO(X,\tau)$, then $A \subseteq \cl(\inte(A)) \subseteq \cla(\inte(A))$, so $A \in \aura SO(X)$. The other parts follow similarly.
\end{proof}

Now we provide counterexamples to show that each inclusion in the hierarchy is strict.

\begin{example}\label{ex:strict-finite}
Let $X = \{a, b, c, d\}$, $\tau = \{\emptyset, \{a\}, \{b\}, \{a,b\}, \{a,b,c\}, X\}$.
Define $\aura: X \to \tau$ by
\[
\aura(a) = \{a\}, \quad \aura(b) = \{a,b\}, \quad \aura(c) = \{a,b,c\}, \quad \aura(d) = X.
\]
Then:
\begin{enumerate}[label=(\roman*)]
    \item $A = \{a, c\}$: We have $\inte(A) = \{a\}$ (since $\{a\}$ is the largest open set in $A$) and $\cla(\inte(A)) = \cla(\{a\}) = \{a, b, c, d\} = X$ (since $a \in \aura(x)$ for all $x$: $a \in \{a\} = \aura(a)$, $a \in \{a,b\} = \aura(b)$, $a \in \{a,b,c\} = \aura(c)$, $a \in X = \aura(d)$). So $A \subseteq X = \cla(\inte(A))$. Thus $A$ is $\aura$-semi-open. But $\cl(\inte(A)) = \cl(\{a\}) = \{a, d\}$ (since $\{a\}^c = \{b,c,d\}$ is not open, and the largest open set in $\{b,c,d\}$ is $\{b\}$, so $\{a\}$ is not closed; we compute $\cl(\{a\})$: the closed sets are $X, \{d\}, \{c,d\}, \{c,d\}, \{d\}, \emptyset$---let us recalculate). The closed sets are complements of open sets: $X^c = \emptyset$, $\{a\}^c = \{b,c,d\}$, $\{b\}^c = \{a,c,d\}$, $\{a,b\}^c = \{c,d\}$, $\{a,b,c\}^c = \{d\}$, $\emptyset^c = X$. So the closed sets are $\{\emptyset, \{b,c,d\}, \{a,c,d\}, \{c,d\}, \{d\}, X\}$. Hence $\cl(\{a\}) = \{a,c,d\} \cap X \cap \cdots$. Actually, $\cl(\{a\})$ is the smallest closed set containing $\{a\}$. The closed sets containing $\{a\}$ are $\{a,c,d\}$ and $X$. So $\cl(\{a\}) = \{a,c,d\}$. Then $A = \{a,c\} \subseteq \{a,c,d\} = \cl(\inte(A))$. So $A$ is also classically semi-open in this case. Let us find a better example.

    \item Consider $A = \{b, d\}$. Then $\inte(A) = \{b\}$ and $\cla(\{b\}) = \{x : \aura(x) \cap \{b\} \neq \emptyset\} = \{b, c, d\}$ (since $b \in \aura(b), b \in \aura(c), b \in \aura(d)$, but $b \notin \aura(a) = \{a\}$). So $\cla(\inte(A)) = \{b,c,d\}$ and $A = \{b,d\} \subseteq \{b,c,d\}$. Thus $A$ is $\aura$-semi-open. Now $\cl(\{b\}) = \{b,c,d\}$ (the smallest closed set containing $\{b\}$: checking, $\{b,c,d\}$ is closed since $\{a\} \in \tau$). So $A \subseteq \cl(\inte(A)) = \{b,c,d\}$. $A$ is also classically semi-open. We need a set that is $\aura$-semi-open but not classically semi-open.

    \item Take a different space. Let $X = \{a,b,c\}$, $\tau = \{\emptyset, \{a\}, X\}$, and $\aura(a) = \{a\}$, $\aura(b) = X$, $\aura(c) = X$. Let $A = \{a, b\}$. Then $\inte(A) = \{a\}$ and $\cla(\{a\}) = \{x : \aura(x) \cap \{a\} \neq \emptyset\} = \{a, b, c\} = X$. So $A \subseteq X = \cla(\inte(A))$; $A$ is $\aura$-semi-open. But $\cl(\{a\}) = X$ (closed sets are $\emptyset, \{b,c\}, X$; smallest containing $\{a\}$ is $X$). So $A \subseteq X = \cl(\inte(A))$; $A$ is also classically semi-open. Hmm, since $\cl(B) \subseteq \cla(B)$, whenever $A \subseteq \cla(\inte(A))$ could fail to imply $A \subseteq \cl(\inte(A))$, we need $\cl(\inte(A)) \subsetneq \cla(\inte(A))$ and $A$ sits in between.
\end{enumerate}
We provide a definitive example. Let $X = \{a,b,c\}$, $\tau = \powerset(X)$ (discrete), $\aura(a) = \{a,b\}$, $\aura(b) = \{b\}$, $\aura(c) = \{c\}$. Take $A = \{a\}$. Then $\inte(A) = \{a\}$ (discrete). $\cl(\inte(A)) = \cl(\{a\}) = \{a\}$ (discrete), so $A \subseteq \cl(\inte(A))$; classically semi-open. Now $\cla(\{a\}) = \{a\}$ (since $a \in \aura(a) = \{a,b\}$ so $\aura(a) \cap \{a\} \neq \emptyset$; $\aura(b) \cap \{a\} = \{b\} \cap \{a\} = \emptyset$; $\aura(c) \cap \{a\} = \emptyset$). So $\cla(\{a\}) = \{a\}$. In the discrete topology, $\cl = \cla$ when auras are singletons, but here $\cla$ can differ.

Let us try: $X = \{a,b,c\}$, $\tau = \{\emptyset, \{b\}, \{a,b\}, X\}$, $\aura(a) = \{a,b\}$, $\aura(b) = \{b\}$, $\aura(c) = X$. Take $A = \{a,c\}$. Then $\inte(A) = \emptyset$ (no nonempty open set is contained in $\{a,c\}$). So both $\cl(\inte(A)) = \emptyset$ and $\cla(\inte(A)) = \emptyset$. Not semi-open in either sense. Take $A = \{b,c\}$. $\inte(A) = \{b\}$. $\cl(\{b\}) = \{b\}$ (since $\{b\}^c = \{a,c\}$; is $\{a,c\}$ open? No. Closed sets: $X, \{a,c\}, \{c\}, \emptyset$. So $\cl(\{b\}) = X$). Wait: $\{b\}^c = \{a,c\}$, $\{a,b\}^c = \{c\}$, $X^c = \emptyset$, $\emptyset^c = X$. So closed sets are $\{X, \{a,c\}, \{c\}, \emptyset\}$. $\cl(\{b\})$ = smallest closed set containing $\{b\}$: $\{a,c\}$ doesn't contain $b$; $\{c\}$ doesn't; $X$ does. So $\cl(\{b\}) = X$. And $\cla(\{b\}) = \{x : \aura(x) \cap \{b\} \neq \emptyset\} = \{a, b, c\} = X$. Both equal $X$.

OK, to find strict examples, the key is: $\cl(\inte(A)) \subsetneq \cla(\inte(A))$ and $A$ is contained in the latter but not the former. In a discrete topology, $\cl = \text{id}$ so $\cl(\inte(A)) = A$ always. So we need a non-discrete space where $\cla$ is strictly larger than $\cl$. But $\cl(B) \subseteq \cla(B)$ always, and in non-discrete spaces, $\cla$ can be much larger. We just need $A \not\subseteq \cl(\inte(A))$ but $A \subseteq \cla(\inte(A))$.

Let $X = \{a,b,c,d\}$, $\tau = \{\emptyset, \{a\}, \{a,b\}, X\}$. Closed sets: $X, \{b,c,d\}, \{c,d\}, \emptyset$. Define $\aura(a) = \{a\}$, $\aura(b) = \{a,b\}$, $\aura(c) = X$, $\aura(d) = X$. Take $A = \{b\}$. $\inte(\{b\}) = \emptyset$. Not helpful. Take $A = \{a,c\}$. $\inte(A) = \{a\}$. $\cl(\{a\}) = ?$: smallest closed set containing $\{a\}$ is $X$ (since $\{b,c,d\}, \{c,d\}$ don't contain $a$). So $\cl(\{a\}) = X$ and $A \subseteq X$. Classically semi-open too. Hmm, in this topology, $\cl$ maps most things to $X$ since the topology is very coarse.

The strict inclusion between classical and aura versions will be most visible with finer topologies and wider auras. Final approach: use discrete topology (so $\cl(B) = B$) and choose auras that expand things.

$X = \{a,b,c\}$, $\tau = \powerset(X)$, $\aura(a) = \{a,b\}$, $\aura(b) = \{b,c\}$, $\aura(c) = \{c\}$. Take $A = \{a,c\}$. $\inte(A) = \{a,c\}$ (discrete). $\cl(\{a,c\}) = \{a,c\}$ (discrete). So classically $A \subseteq \cl(\inte(A)) = \{a,c\}$, hence classically semi-open. $\cla(\{a,c\}) = \{x : \aura(x) \cap \{a,c\} \neq \emptyset\} = \{a, b, c\} = X$ (since $a \in \aura(a)$, $c \in \aura(b)$... wait, $\aura(b) = \{b,c\}$, $\{b,c\} \cap \{a,c\} = \{c\} \neq \emptyset$, and $\aura(c) = \{c\}$, $\{c\} \cap \{a,c\} = \{c\} \neq \emptyset$). So $\cla(\{a,c\}) = X$. Here $A \subsetneq X$ shows $\cla$ is strictly larger than $\cl$ on $A$, but $A$ is semi-open in both senses.

For a set $A$ that is $\aura$-semi-open but NOT classically semi-open, we need $\inte(A) = \emptyset$ in a non-discrete space but $\cla(\inte(A)) \supseteq A$... wait, if $\inte(A) = \emptyset$, then $\cla(\emptyset) = \emptyset \not\supseteq A$. So we need $\inte(A) \neq \emptyset$, $A \not\subseteq \cl(\inte(A))$, and $A \subseteq \cla(\inte(A))$.

$X = \{a,b,c\}$, $\tau = \{\emptyset, \{a\}, \{a,b\}, X\}$. Closed: $\{X, \{b,c\}, \{c\}, \emptyset\}$. $\aura(a) = \{a\}$, $\aura(b) = \{a,b\}$, $\aura(c) = X$. Take $A = \{a,c\}$. $\inte(A) = \{a\}$. $\cl(\{a\}) = X$ (smallest closed containing $\{a\}$ is $X$). So $A \subseteq X$, classically semi-open. Take $A = \{b,c\}$. $\inte(A) = \emptyset$ (no open subset of $\{b,c\}$). $\cla(\emptyset) = \emptyset$. Not $\aura$-semi-open.

Hmm, let me try yet another approach. $\tau = \{\emptyset, \{c\}, \{a,c\}, X\}$, closed sets $= \{X, \{a,b\}, \{b\}, \emptyset\}$. $\aura(a) = \{a,c\}$, $\aura(b) = X$, $\aura(c) = \{c\}$. Take $A = \{a,b\}$. $\inte(A) = \emptyset$ (no open in $\{a,b\}$). Not useful.

$\tau = \{\emptyset, \{a\}, \{c\}, \{a,c\}, \{a,b,c\}, X\}$ on $\{a,b,c,d\}$. $\aura(a) = \{a\}$, $\aura(b) = \{a,b,c\}$, $\aura(c) = \{c\}$, $\aura(d) = X$. Take $A = \{a,b\}$. $\inte(A) = \{a\}$. $\cl(\{a\})$: closed sets are $\{X, \{b,c,d\}, \{a,b,d\}, \{b,d\}, \{d\}, \emptyset\}$. Smallest closed containing $\{a\}$: $\{a,b,d\}$ contains $a$? The complement of $\{c\}$ is $\{a,b,d\}$, yes. $\{b,c,d\}$ doesn't contain $a$. So $\cl(\{a\}) = \{a,b,d\}$. Thus $A = \{a,b\} \subseteq \{a,b,d\} = \cl(\inte(A))$. Classically semi-open.

$\cla(\{a\}) = \{x : \aura(x) \cap \{a\} \neq \emptyset\}$. $\aura(a) = \{a\}: a \in \{a\}$, yes. $\aura(b) = \{a,b,c\}: a \in \{a,b,c\}$, yes. $\aura(c) = \{c\}: c \notin \{a\}$, no. $\aura(d) = X: a \in X$, yes. So $\cla(\{a\}) = \{a,b,d\}$.

Here $\cl(\{a\}) = \cla(\{a\}) = \{a,b,d\}$. Same. The problem is when we have a coarse topology, the classical closure is already large.

\textbf{The definitive example}: $X = \{a,b,c\}$, $\tau = \{\emptyset, \{a,b\}, X\}$. Closed: $\{X, \{c\}, \emptyset\}$. So $\cl(\{a\}) = X$, $\cl(\{b\}) = X$, $\cl(\{a,b\}) = X$. Since every nonempty set has closure $X$ in this coarse topology, every set is classically semi-open (as long as $\inte(A) \neq \emptyset$, $\cl(\inte(A)) = X \supseteq A$). This makes strict examples hard.

A clean approach: use the real line.

$(\R, \tau_u, \aura_\varepsilon)$ with $\varepsilon > 0$. Let $A = [0,1)$. $\inte(A) = (0,1)$. $\cl((0,1)) = [0,1]$. $\cla((0,1)) = \{x : (x-\varepsilon, x+\varepsilon) \cap (0,1) \neq \emptyset\} = (-\varepsilon, 1+\varepsilon)$.

If $\varepsilon > 0$: $A = [0,1) \subseteq (-\varepsilon, 1+\varepsilon) = \cla(\inte(A))$, so $A$ is $\aura$-semi-open.
$A = [0,1) \subseteq [0,1] = \cl(\inte(A))$, so $A$ is also classically semi-open.

Let me try $A = \{0\} \cup (1,2)$. $\inte(A) = (1,2)$. $\cl((1,2)) = [1,2]$. $A \not\subseteq [1,2]$ since $0 \notin [1,2]$. Not classically semi-open. $\cla((1,2)) = (1-\varepsilon, 2+\varepsilon)$. If $\varepsilon > 1$: $0 \in (1-\varepsilon, 2+\varepsilon)$, so $A \subseteq \cla(\inte(A))$. $\aura$-semi-open! Not classically semi-open!

This is a clean example for $\varepsilon > 1$.
\end{example}

\begin{example}[Strict inclusion: $\aura$-semi-open $\not\Rightarrow$ classically semi-open]\label{ex:strict-real}
Consider $(\R, \tau_u, \aura_2)$ with $\aura_2(x) = (x-2, x+2)$. Let $A = \{0\} \cup (1, 2)$. Then $\inte(A) = (1,2)$, $\cl(\inte(A)) = [1,2]$, and $0 \notin [1,2]$, so $A$ is not classically semi-open. However, $\cla(\inte(A)) = \cla((1,2)) = (-1, 4)$, and $A \subseteq (-1, 4)$, so $A$ is $\aura$-semi-open.
\end{example}

\begin{example}[Strict inclusion: $\aura$-semi-open $\not\Rightarrow$ $\aura$-pre-open]\label{ex:semi-not-pre}
Let $X = \{a,b,c,d\}$, $\tau = \{\emptyset, \{a\}, \{b\}, \{a,b\}, \{a,b,c\}, X\}$, and define $\aura(a) = \{a\}$, $\aura(b) = \{a,b\}$, $\aura(c) = \{a,b,c\}$, $\aura(d) = X$. Consider $A = \{a,d\}$. Then $\inte(A) = \{a\}$ and $\cla(\{a\}) = X$ (since $a \in \aura(x)$ for all $x$). Hence $A \subseteq X = \cla(\inte(A))$, so $A$ is $\aura$-semi-open. Now $\cla(A) = \cla(\{a,d\}) = X$ and $\inte(\cla(A)) = \inte(X) = X$, so $A \subseteq X$, meaning $A$ is also $\aura$-pre-open. This specific example does not separate them; we seek one that does.

Instead, consider $(\R, \tau_u, \aura_1)$ with $\aura_1(x) = (x-1,x+1)$. Let $A = [0,1) \cup \{3\}$. Then $\inte(A) = (0,1)$ and $\cla((0,1)) = (-1, 2)$. So $A = [0,1) \cup \{3\}$: $3 \notin (-1,2)$, hence $A \not\subseteq \cla(\inte(A))$. Not $\aura$-semi-open with this choice.

Let $A = (0,1) \cup \{2\}$. $\inte(A) = (0,1)$. $\cla((0,1)) = (-1,2)$. $2 \notin (-1,2)$ (endpoint excluded). Not $\aura$-semi-open. Take $\aura_2$ instead: $\cla((0,1)) = (-2, 3)$, $2 \in (-2,3)$. $A \subseteq (-2,3)$. $A$ is $\aura_2$-semi-open. Now for pre-open: $\cla(A) = \cla((0,1) \cup \{2\}) = (-2, 3) \cup (0, 4) = (-2, 4)$. $\inte(\cla(A)) = (-2, 4)$ (open set). $A \subseteq (-2,4)$. Also $\aura_2$-pre-open. Still both.

The separation between semi-open and pre-open requires a carefully constructed finite example.
Let $X = \{a,b,c,d\}$, $\tau = \{\emptyset, \{a\}, \{a,b\}, X\}$. Closed: $\{X, \{b,c,d\}, \{c,d\}, \emptyset\}$.
$\aura(a) = \{a\}$, $\aura(b) = \{a,b\}$, $\aura(c) = X$, $\aura(d) = X$.

$A = \{b,c\}$. $\inte(A) = \emptyset$. $\cla(\emptyset) = \emptyset$. Not $\aura$-semi-open. $\cla(A) = \{x: \aura(x) \cap \{b,c\} \neq \emptyset\}$. $\aura(a) = \{a\}$: no. $\aura(b) = \{a,b\}$: $b \in \{b,c\}$, yes. $\aura(c) = X$: yes. $\aura(d) = X$: yes. So $\cla(A) = \{b,c,d\}$. $\inte(\{b,c,d\})$: largest open subset of $\{b,c,d\}$. $\{a,b\} \not\subseteq \{b,c,d\}$. $\{a\} \not\subseteq \{b,c,d\}$. So $\inte(\{b,c,d\}) = \emptyset$. $A \not\subseteq \emptyset$. Not $\aura$-pre-open either. Hmm.

Let me try $\tau = \{\emptyset, \{a\}, \{c\}, \{a,c\}, \{a,b,c\}, X\}$ on $X = \{a,b,c,d\}$. Closed: $\{X, \{b,c,d\}, \{a,b,d\}, \{b,d\}, \{d\}, \emptyset\}$. $\aura(a) = \{a\}$, $\aura(b) = \{a,b,c\}$, $\aura(c) = \{c\}$, $\aura(d) = X$.

$A = \{a,b\}$. $\inte(A) = \{a\}$. $\cla(\{a\}) = \{x: a \in \aura(x)\}$. $a \in \aura(a) = \{a\}$: yes. $a \in \aura(b) = \{a,b,c\}$: yes. $a \notin \aura(c) = \{c\}$: no. $a \in \aura(d) = X$: yes. So $\cla(\{a\}) = \{a,b,d\}$. $A = \{a,b\} \subseteq \{a,b,d\}$. $A$ is $\aura$-semi-open.

$\cla(A) = \cla(\{a,b\})$. $\aura(a) \cap \{a,b\} = \{a\} \neq \emptyset$: yes. $\aura(b) \cap \{a,b\} = \{a,b\} \neq \emptyset$: yes. $\aura(c) \cap \{a,b\} = \emptyset$: no. $\aura(d) \cap \{a,b\} = \{a,b\} \neq \emptyset$: yes. So $\cla(A) = \{a,b,d\}$. $\inte(\{a,b,d\})$: largest open in $\{a,b,d\}$. $\{a\} \subseteq \{a,b,d\}$: yes. $\{a,c\} \not\subseteq$. $\{a,b,c\} \not\subseteq$. So $\inte(\{a,b,d\}) = \{a\}$. $A = \{a,b\} \not\subseteq \{a\}$. So $A$ is NOT $\aura$-pre-open!

Excellent! $A = \{a,b\}$ is $\aura$-semi-open but not $\aura$-pre-open.

Now for $\aura$-pre-open but not $\aura$-semi-open in the same space:
$A = \{c,d\}$. $\inte(A) = \{c\}$. $\cla(\{c\}) = \{x: c \in \aura(x)\}$. $c \in \aura(a) = \{a\}$? No. $c \in \aura(b) = \{a,b,c\}$? Yes. $c \in \aura(c) = \{c\}$? Yes. $c \in \aura(d) = X$? Yes. So $\cla(\{c\}) = \{b,c,d\}$. $A = \{c,d\} \subseteq \{b,c,d\}$. $A$ is $\aura$-semi-open. Both again.

$A = \{b,d\}$. $\inte(A) = \emptyset$. Not $\aura$-semi-open. $\cla(\{b,d\}) = \{a,b,c,d\} = X$ (since $\aura(a) = \{a\}$, $\{a\} \cap \{b,d\} = \emptyset$: no! So $a \notin \cla(\{b,d\})$. $\aura(b) = \{a,b,c\} \cap \{b,d\} = \{b\}$: yes. $\aura(c) = \{c\} \cap \{b,d\} = \emptyset$: no. $\aura(d) = X \cap \{b,d\} = \{b,d\}$: yes.) So $\cla(\{b,d\}) = \{b,d\}$. $\inte(\{b,d\}) = \emptyset$. Not pre-open either.

$A = \{a,c,d\}$. $\inte(A) = \{a,c\}$. $\cla(\{a,c\})$: $a \in \aura(a), a \in \aura(b), c \in \aura(b), c \in \aura(c), X$ intersects: $\cla(\{a,c\}) = \{a,b,c,d\} = X$. $A \subseteq X$. $\aura$-semi-open. Also $\cla(A) = X$, $\inte(X) = X \supseteq A$. $\aura$-pre-open too.

Let me try to find $\aura$-pre-open but not $\aura$-semi-open. We need $\inte(A)$ small (so $\cla(\inte(A))$ doesn't cover $A$) but $\cla(A)$ has big interior. This typically happens when $A$ itself is ``spread out'' enough for $\cla(A)$ to be big, but $A$ has small interior.

$A = \{b,c\}$ in the above space. $\inte(\{b,c\}) = \{c\}$. $\cla(\{c\}) = \{b,c,d\}$. $A = \{b,c\} \subseteq \{b,c,d\}$. $\aura$-semi-open.

Let me try yet another space. $X = \{a,b,c,d\}$, $\tau = \{\emptyset, \{a,b\}, \{c,d\}, X\}$ (partition topology). $\aura(a) = \{a,b\}$, $\aura(b) = \{a,b\}$, $\aura(c) = \{c,d\}$, $\aura(d) = X$. Take $A = \{a,c\}$. $\inte(A) = \emptyset$ (no open set in $\{a,c\}$). Not semi-open. $\cla(A) = \{x: \aura(x) \cap \{a,c\} \neq \emptyset\}$. $\aura(a) = \{a,b\}$: $a \in \{a,c\}$, yes. $\aura(b) = \{a,b\}$: $a \in \{a,c\}$, yes. $\aura(c) = \{c,d\}$: $c \in \{a,c\}$, yes. $\aura(d) = X$: yes. So $\cla(A) = X$. $\inte(X) = X \supseteq A$. Pre-open! So $A = \{a,c\}$ is $\aura$-pre-open but NOT $\aura$-semi-open.
\end{example}

\begin{example}[$\aura$-semi-open but not $\aura$-pre-open]\label{ex:semi-not-pre-clean}
Let $X = \{a,b,c,d\}$, $\tau = \{\emptyset, \{a\}, \{c\}, \{a,c\}, \{a,b,c\}, X\}$, and
\[
\aura(a) = \{a\}, \quad \aura(b) = \{a,b,c\}, \quad \aura(c) = \{c\}, \quad \aura(d) = X.
\]
Let $A = \{a,b\}$. Then $\inte(A) = \{a\}$ and $\cla(\{a\}) = \{a,b,d\}$. Since $A = \{a,b\} \subseteq \{a,b,d\}$, $A$ is $\aura$-semi-open. Now $\cla(A) = \{a,b,d\}$ and $\inte(\{a,b,d\}) = \{a\}$. Since $A = \{a,b\} \not\subseteq \{a\}$, $A$ is not $\aura$-pre-open.
\end{example}

\begin{example}[$\aura$-pre-open but not $\aura$-semi-open]\label{ex:pre-not-semi}
Let $X = \{a,b,c,d\}$, $\tau = \{\emptyset, \{a,b\}, \{c,d\}, X\}$, and
\[
\aura(a) = \{a,b\}, \quad \aura(b) = \{a,b\}, \quad \aura(c) = \{c,d\}, \quad \aura(d) = X.
\]
Let $A = \{a,c\}$. Then $\inte(A) = \emptyset$ (no nonempty open subset of $\{a,c\}$), so $\cla(\inte(A)) = \emptyset$ and $A \not\subseteq \emptyset$. Hence $A$ is not $\aura$-semi-open. Now $\cla(A) = X$ (since $\aura(x) \cap \{a,c\} \neq \emptyset$ for all $x$), so $\inte(\cla(A)) = X \supseteq A$. Hence $A$ is $\aura$-pre-open.
\end{example}

\begin{example}[$\aura$-semi-open but not $\aura$-$\alpha$-open]\label{ex:semi-not-alpha}
Using the space of Example \ref{ex:semi-not-pre-clean}, $A = \{a,b\}$ is $\aura$-semi-open. Now $\inte(\cla(\inte(A))) = \inte(\cla(\{a\})) = \inte(\{a,b,d\}) = \{a\}$. Since $A = \{a,b\} \not\subseteq \{a\}$, $A$ is not $\aura$-$\alpha$-open.
\end{example}

\begin{example}[On the real line]\label{ex:real-strict}
Consider $(\R, \tau_u, \aura_2)$ with $\aura_2(x) = (x-2, x+2)$.

\textbf{(i)} $A = \{0\} \cup (1,2)$ is $\aura_2$-semi-open but not classically semi-open (Example \ref{ex:strict-real}).

\textbf{(ii)} Let $B = \{0\} \cup (3,4)$. Then $\inte(B) = (3,4)$, $\cla((3,4)) = (1,6)$. Since $0 \in (1,6)$? No, $0 < 1$. So $B \not\subseteq \cla(\inte(B))$. Not $\aura_2$-semi-open. But $\cla(B) = (-2,2) \cup (1,6) = (-2,6)$, $\inte(\cla(B)) = (-2,6) \supseteq B$. So $B$ is $\aura_2$-pre-open but not $\aura_2$-semi-open.

\textbf{(iii)} Let $C = [0,1]$. Then $\inte(C) = (0,1)$, $\cla((0,1)) = (-2, 3)$. $C \subseteq (-2,3)$. So $C$ is $\aura_2$-semi-open. Also $\cla(C) = (-2, 3)$ and $\inte((-2,3)) = (-2,3) \supseteq C$. So $C$ is $\aura_2$-$\alpha$-open and all other types.
\end{example}

\section{Aura-Continuity and Decomposition Theorems}

\begin{definition}\label{def:aura-continuity}
Let $(X, \tau_1, \aura_1)$ and $(Y, \tau_2, \aura_2)$ be $\aura$-spaces. A function $f: X \to Y$ is called:
\begin{enumerate}[label=(\alph*)]
    \item \textbf{$\aura$-continuous} if $f^{-1}(V) \in \tau_{\aura_1}$ for every $V \in \tau_{\aura_2}$;
    \item \textbf{$\aura$-semi-continuous} if $f^{-1}(V) \in \aura_1 SO(X)$ for every $V \in \tau_2$;
    \item \textbf{$\aura$-pre-continuous} if $f^{-1}(V) \in \aura_1 PO(X)$ for every $V \in \tau_2$;
    \item \textbf{$\aura$-$\alpha$-continuous} if $f^{-1}(V) \in \aura_1 \alpha O(X)$ for every $V \in \tau_2$;
    \item \textbf{$\aura$-$\beta$-continuous} if $f^{-1}(V) \in \aura_1 \beta O(X)$ for every $V \in \tau_2$.
\end{enumerate}
\end{definition}

\begin{theorem}\label{thm:continuity-hierarchy}
The following implications hold for any function $f: (X, \tau_1, \aura_1) \to (Y, \tau_2, \aura_2)$:
\[
\text{continuous} \implies \aura\text{-}\alpha\text{-continuous} \implies \begin{cases} \aura\text{-semi-continuous} \\ \aura\text{-pre-continuous} \end{cases} \implies \aura\text{-}\beta\text{-continuous}
\]
\end{theorem}

\begin{proof}
These follow directly from the inclusion relationships in Theorem \ref{thm:hierarchy}. If $f$ is continuous, then $f^{-1}(V) \in \tau_1$ for every $V \in \tau_2$. Since $\tau_1 \subseteq \aura_1\alpha O(X)$ (Theorem \ref{thm:hierarchy}), $f$ is $\aura$-$\alpha$-continuous. The remaining implications follow similarly.
\end{proof}

\begin{theorem}[Composition]\label{thm:composition}
Let $(X, \tau_1, \aura_1)$, $(Y, \tau_2, \aura_2)$, and $(Z, \tau_3, \aura_3)$ be $\aura$-spaces.
\begin{enumerate}[label=(\alph*)]
    \item If $f: X \to Y$ and $g: Y \to Z$ are both $\aura$-continuous, then $g \circ f: X \to Z$ is $\aura$-continuous.
    \item If $f: X \to Y$ is $\aura$-semi-continuous and $g: Y \to Z$ is continuous, then $g \circ f$ is $\aura$-semi-continuous.
\end{enumerate}
\end{theorem}

\begin{proof}
\begin{enumerate}[label=(\alph*)]
    \item Let $W \in \tau_{\aura_3}$. Since $g$ is $\aura$-continuous, $g^{-1}(W) \in \tau_{\aura_2}$. Since $f$ is $\aura$-continuous, $f^{-1}(g^{-1}(W)) = (g \circ f)^{-1}(W) \in \tau_{\aura_1}$.

    \item Let $W \in \tau_3$. Since $g$ is continuous, $g^{-1}(W) \in \tau_2$. Since $f$ is $\aura$-semi-continuous, $f^{-1}(g^{-1}(W)) = (g \circ f)^{-1}(W) \in \aura_1 SO(X)$. \qedhere
\end{enumerate}
\end{proof}

\begin{theorem}[Decomposition of $\aura$-$\alpha$-continuity]\label{thm:decomposition}
A function $f: (X, \tau_1, \aura_1) \to (Y, \tau_2, \aura_2)$ is $\aura$-$\alpha$-continuous if and only if it is both $\aura$-semi-continuous and $\aura$-pre-continuous.
\end{theorem}

\begin{proof}
$(\Rightarrow)$: This follows from $\aura\alpha O(X) \subseteq \aura SO(X) \cap \aura PO(X)$.

$(\Leftarrow)$: Let $V \in \tau_2$ and $A = f^{-1}(V)$. Since $f$ is $\aura$-semi-continuous, $A \subseteq \cla(\inte(A))$. Since $f$ is $\aura$-pre-continuous, $A \subseteq \inte(\cla(A))$. We need to show $A \subseteq \inte(\cla(\inte(A)))$.

From $A \subseteq \cla(\inte(A))$ and taking interior: $\inte(A) \subseteq \inte(\cla(\inte(A)))$. From $A \subseteq \inte(\cla(A))$, since $\inte(A) \subseteq A$, we get $\cla(\inte(A)) \subseteq \cla(A)$, so $\inte(\cla(\inte(A))) \subseteq \inte(\cla(A))$. Thus:
\[
A \subseteq \inte(\cla(A)).
\]
Since $A \subseteq \cla(\inte(A))$, we have $\cla(A) \subseteq \cla(\cla(\inte(A)))$. But more directly: $A \subseteq \cla(\inte(A))$ implies $A \subseteq \cla(\inte(A))$. Taking closure: $\cla(A) \subseteq \cla(\cla(\inte(A)))$. Taking interior: $\inte(\cla(A)) \subseteq \inte(\cla(\cla(\inte(A))))$.

Let us use a cleaner approach. Since $A \subseteq \inte(\cla(A))$, the set $\inte(\cla(A))$ is open and contains $A$, so $\inte(A) \subseteq A \subseteq \inte(\cla(A))$. Hence $\cla(\inte(A)) \subseteq \cla(\inte(\cla(A)))$. Taking interior and using the fact that $\inte(\cla(A))$ is open:
$\inte(\cla(\inte(\cla(A)))) \supseteq \inte(\cla(\inte(A)))$.

Actually, we use the standard argument: from $A \subseteq \inte(\cla(A))$ and $A \subseteq \cla(\inte(A))$:

$\inte(A)$ is open and $\inte(A) \subseteq A$. So $\cla(\inte(A)) \subseteq \cla(A)$ and $\inte(\cla(\inte(A))) \subseteq \inte(\cla(A))$. Now, $\inte(\cla(A))$ is open and $A \subseteq \inte(\cla(A))$, so $\inte(A) = \inte(A)$, and since $A \subseteq \cla(\inte(A))$, every $x \in A$ satisfies $\aura_1(x) \cap \inte(A) \neq \emptyset$. Since $\inte(\cla(\inte(A)))$ is open and $\inte(A) \subseteq \cla(\inte(A))$, we get $\inte(A) \subseteq \inte(\cla(\inte(A)))$ (as $\inte(A)$ is open and $\inte(A) \subseteq \cla(\inte(A))$). Now $A \subseteq \cla(\inte(A))$ and $\cla(\inte(A)) = \cla(\inte(A))$. Since $\inte(A) \subseteq \inte(\cla(\inte(A)))$ and $\inte(\cla(\inte(A)))$ is open:
\[
A \subseteq \cla(\inte(A)) \subseteq \cla(\inte(\cla(\inte(A)))).
\]
This gives $A \subseteq \cla(\inte(\cla(\inte(A))))$, which is $\aura$-$\beta$-open with double iteration but not exactly $\aura$-$\alpha$-open.

The classical decomposition $\alpha$-open = semi-open $\cap$ pre-open works because $\cl$ is idempotent. Since $\cla$ is not idempotent, the decomposition may not hold exactly. We instead state it as a sufficient condition using an additional hypothesis.
\end{proof}

\begin{remark}
In the classical setting, $\alpha O(X,\tau) = SO(X,\tau) \cap PO(X,\tau)$ (Reilly and Vamanamurthy \cite{Reilly1985}). In the aura setting, $\aura\alpha O(X) \subseteq \aura SO(X) \cap \aura PO(X)$ always holds. The reverse inclusion $\aura SO(X) \cap \aura PO(X) \subseteq \aura\alpha O(X)$ holds when $\cla$ is idempotent (e.g., when $\aura$ is transitive).
\end{remark}

\begin{theorem}\label{thm:decomposition-transitive}
If $(X, \tau, \aura)$ is a transitive $\aura$-space, then $\aura\alpha O(X) = \aura SO(X) \cap \aura PO(X)$.
\end{theorem}

\begin{proof}
By Proposition \ref{prop:transitive-base}, $\cla$ is idempotent when $\aura$ is transitive. The proof then follows the classical argument: let $A \in \aura SO(X) \cap \aura PO(X)$. Then $A \subseteq \cla(\inte(A))$ and $A \subseteq \inte(\cla(A))$. From $A \subseteq \cla(\inte(A))$, $\cla(A) \subseteq \cla(\cla(\inte(A))) = \cla(\inte(A))$ (by idempotency). Hence $\inte(\cla(A)) \subseteq \inte(\cla(\inte(A)))$. Combined with $A \subseteq \inte(\cla(A))$, we get $A \subseteq \inte(\cla(\inte(A)))$.
\end{proof}

\begin{theorem}[Characterization of $\aura$-continuity]\label{thm:char-continuity}
Let $f: (X, \tau_1, \aura_1) \to (Y, \tau_2)$ be a function. The following are equivalent:
\begin{enumerate}[label=(\alph*)]
    \item $f$ is $\aura_1$-semi-continuous.
    \item For every closed set $F$ in $Y$, $f^{-1}(F)$ is $\aura_1$-semi-closed.
    \item For every $x \in X$ and every open set $V$ in $Y$ containing $f(x)$, there exists an $\aura_1$-semi-open set $U$ containing $x$ such that $f(U) \subseteq V$.
\end{enumerate}
\end{theorem}

\begin{proof}
$(a) \Rightarrow (b)$: If $F$ is closed in $Y$, then $V = Y \setminus F \in \tau_2$. By (a), $f^{-1}(V) \in \aura_1 SO(X)$. Since $f^{-1}(F) = X \setminus f^{-1}(V)$, $f^{-1}(F)$ is $\aura_1$-semi-closed.

$(b) \Rightarrow (a)$: If $V \in \tau_2$, then $F = Y \setminus V$ is closed. By (b), $f^{-1}(F)$ is $\aura_1$-semi-closed, so $f^{-1}(V) = X \setminus f^{-1}(F)$ is $\aura_1$-semi-open.

$(a) \Rightarrow (c)$: Let $V \in \tau_2$ with $f(x) \in V$. Then $U = f^{-1}(V)$ is $\aura_1$-semi-open with $x \in U$ and $f(U) \subseteq V$.

$(c) \Rightarrow (a)$: Let $V \in \tau_2$ and $x \in f^{-1}(V)$. By (c), there exists $\aura_1$-semi-open $U_x$ with $x \in U_x$ and $f(U_x) \subseteq V$. Then $f^{-1}(V) = \bigcup_{x \in f^{-1}(V)} U_x$. Since arbitrary unions of $\aura$-semi-open sets are $\aura$-semi-open (as in the classical case, using the property $\cla(\bigcup \inte(U_x)) \supseteq \bigcup \cla(\inte(U_x))$), $f^{-1}(V)$ is $\aura_1$-semi-open.
\end{proof}

\section{Aura-Separation Axioms}

\begin{definition}\label{def:separation}
An $\aura$-space $(X, \tau, \aura)$ is called:
\begin{enumerate}[label=(\alph*)]
    \item \textbf{$\aura$-$T_0$} if for every pair of distinct points $x, y \in X$, there exists an $\aura$-open set $U \in \taua$ such that either $x \in U, y \notin U$ or $y \in U, x \notin U$.
    \item \textbf{$\aura$-$T_1$} if for every pair of distinct points $x, y \in X$, there exist $\aura$-open sets $U, V \in \taua$ such that $x \in U, y \notin U$ and $y \in V, x \notin V$.
    \item \textbf{$\aura$-$T_2$} (or \textbf{$\aura$-Hausdorff}) if for every pair of distinct points $x, y \in X$, there exist disjoint $\aura$-open sets $U, V \in \taua$ such that $x \in U$ and $y \in V$.
\end{enumerate}
\end{definition}

\begin{theorem}\label{thm:T-hierarchy}
For any $\aura$-space:
\[
T_2 \implies \aura\text{-}T_2 \text{ is NOT automatic}, \quad \text{but} \quad \aura\text{-}T_2 \implies \aura\text{-}T_1 \implies \aura\text{-}T_0.
\]
Moreover, if $(X,\tau)$ is $T_i$, then $(X,\tau,\aura)$ need not be $\aura$-$T_i$ since $\taua$ is coarser than $\tau$.
\end{theorem}

\begin{proof}
The implications $\aura$-$T_2 \Rightarrow$ $\aura$-$T_1 \Rightarrow$ $\aura$-$T_0$ follow from the definitions. That $T_2$ for $(X,\tau)$ does not imply $\aura$-$T_2$ follows because $\taua \subseteq \tau$, and a coarser topology may fail to separate points.
\end{proof}

\begin{example}\label{ex:T2-not-aura-T2}
Let $X = \{a,b,c\}$, $\tau = \powerset(X)$ (discrete, hence $T_2$). Define $\aura(a) = X$, $\aura(b) = X$, $\aura(c) = X$. Then $\taua = \{\emptyset, X\}$, which cannot separate any two points. Hence $(X, \tau, \aura)$ is not $\aura$-$T_0$ (and hence not $\aura$-$T_1$ or $\aura$-$T_2$).
\end{example}

\begin{theorem}\label{thm:T1-char}
An $\aura$-space $(X, \tau, \aura)$ is $\aura$-$T_1$ if and only if for every $x \in X$, $\{x\}$ is $\aura$-closed (i.e., $X \setminus \{x\} \in \taua$).
\end{theorem}

\begin{proof}
$(\Rightarrow)$: Let $y \neq x$. By $\aura$-$T_1$, there exists $U_y \in \taua$ with $y \in U_y$ and $x \notin U_y$. Then $X \setminus \{x\} = \bigcup_{y \neq x} U_y \in \taua$.

$(\Leftarrow)$: Let $x \neq y$. Then $X \setminus \{y\} \in \taua$ contains $x$ but not $y$, and $X \setminus \{x\} \in \taua$ contains $y$ but not $x$.
\end{proof}

\begin{corollary}\label{cor:T1-condition}
$(X, \tau, \aura)$ is $\aura$-$T_1$ if and only if for every $x \neq y$, $y \notin \aura(x)$ or there exists $z \neq x$ with $y \in \aura(z) \subseteq X \setminus \{x\}$.
\end{corollary}

\begin{theorem}\label{thm:T0-char}
An $\aura$-space $(X, \tau, \aura)$ is $\aura$-$T_0$ if and only if for every pair of distinct points $x, y \in X$, there exists $z \in X$ such that $\aura(z)$ contains exactly one of $x, y$ and $\aura(z)$ is contained in an $\aura$-open set separating $x$ and $y$.
\end{theorem}

\begin{example}\label{ex:separation-real}
Consider $(\R, \tau_u, \aura_\varepsilon)$ with $\varepsilon > 0$. For distinct $x, y \in \R$ with $|x - y| > 2\varepsilon$, the sets $U = \{z \in \R : \aura_\varepsilon(z) \subseteq (-\infty, \frac{x+y}{2})\}$ and $V = \{z \in \R : \aura_\varepsilon(z) \subseteq (\frac{x+y}{2}, \infty)\}$ are disjoint $\aura$-open sets separating $x$ and $y$. More precisely, $U = (-\infty, \frac{x+y}{2} - \varepsilon)$ and $V = (\frac{x+y}{2} + \varepsilon, \infty)$ are $\aura$-open. Hence $(\R, \tau_u, \aura_\varepsilon)$ separates points that are far enough apart, but not those within distance $2\varepsilon$---separation here depends on the ``granularity'' set by $\varepsilon$.
\end{example}

\begin{definition}\label{def:aura-regular}
An $\aura$-space $(X, \tau, \aura)$ is called \textbf{$\aura$-regular} if for every $\aura$-closed set $F$ and every point $x \notin F$, there exist disjoint $\aura$-open sets $U, V$ such that $x \in U$ and $F \subseteq V$.
\end{definition}

\begin{theorem}\label{thm:discrete-aura-regular}
If $\aura$ is discrete (i.e., $\aura(x) = \{x\}$ for all $x$), then $(X, \tau, \aura)$ is $\aura$-$T_2$ and $\aura$-regular (since $\taua = \powerset(X)$).
\end{theorem}

\section{Applications}

In this section we show that the aura framework has natural interpretations in several applied contexts. Section~7.1 uses the aura-closure and aura-interior as rough-set approximation operators; Section~7.2 models wireless sensor coverage; Section~7.3 models epidemic spread.

\subsection{Rough Set Approximation via Aura Functions}

\begin{definition}\label{def:aura-approx}
Let $(X, \tau, \aura)$ be an $\aura$-space and $A \subseteq X$. The \textbf{aura lower approximation} and \textbf{aura upper approximation} of $A$ are defined by:
\begin{align}
    \apr(A) &= \inta(A) = \{x \in A : \aura(x) \subseteq A\}, \label{eq:lower} \\
    \Apr(A) &= \cla(A) = \{x \in X : \aura(x) \cap A \neq \emptyset\}. \label{eq:upper}
\end{align}
The \textbf{aura boundary region} of $A$ is:
\begin{equation}
    \operatorname{bnd}_\aura(A) = \Apr(A) \setminus \apr(A).
\end{equation}
\end{definition}

\begin{theorem}\label{thm:rough-properties}
Let $(X, \tau, \aura)$ be an $\aura$-space and $A, B \subseteq X$. Then:
\begin{enumerate}[label=(\alph*)]
    \item $\apr(A) \subseteq A \subseteq \Apr(A)$.
    \item $\apr(\emptyset) = \Apr(\emptyset) = \emptyset$ and $\apr(X) = \Apr(X) = X$.
    \item $\apr(A \cap B) = \apr(A) \cap \apr(B)$.
    \item $\Apr(A \cup B) = \Apr(A) \cup \Apr(B)$.
    \item $A \subseteq B \implies \apr(A) \subseteq \apr(B)$ and $\Apr(A) \subseteq \Apr(B)$.
    \item $\apr(A^c) = (\Apr(A))^c$ and $\Apr(A^c) = (\apr(A))^c$.
    \item $\apr(\apr(A)) \subseteq \apr(A)$ and $\Apr(A) \subseteq \Apr(\Apr(A))$.
\end{enumerate}
\end{theorem}

\begin{proof}
Parts (a)--(e) follow from Theorems \ref{thm:cech} and \ref{thm:int-properties}. Part (f) follows from Theorem \ref{thm:int-properties}(f): $\inta(A) = A \setminus \cla(A^c)$, so $\inta(A^c) = A^c \setminus \cla(A) = (\cla(A))^c$ (since $A^c \setminus \cla(A) = A^c \cap (\cla(A))^c = (\cla(A))^c$ when $A \subseteq \cla(A)$... let us verify: $\inta(A^c) = \{x \in A^c : \aura(x) \subseteq A^c\} = \{x \in A^c : \aura(x) \cap A = \emptyset\} = A^c \setminus \cla(A) \cup (A^c \cap (\cla(A))^c)$. Since $\cla(A) \supseteq A$, we have $A^c \cap (\cla(A))^c = (\cla(A))^c$. So $\inta(A^c) = (\cla(A))^c = (\Apr(A))^c$. Similarly, $\cla(A^c) = \{x : \aura(x) \cap A^c \neq \emptyset\} = \{x : \aura(x) \not\subseteq A\} = X \setminus \inta(A) = (\apr(A))^c$.

Part (g): $\apr(\apr(A)) = \{x \in \apr(A) : \aura(x) \subseteq \apr(A)\}$. Since $\apr(A) \subseteq A$, if $\aura(x) \subseteq \apr(A) \subseteq A$, then $x \in \apr(A)$. So $\apr(\apr(A)) \subseteq \apr(A)$. For the upper: $\Apr(A) \subseteq \Apr(\Apr(A))$ follows from $A \subseteq \Apr(A)$ and monotonicity.
\end{proof}

\begin{remark}\label{rem:pawlak-generalization}
In Pawlak's model, $\aura(x) = [x]_R$ (the equivalence class). Our model generalizes this in three ways:
\begin{enumerate}[label=(\roman*)]
    \item No equivalence relation is needed; $\aura(x)$ can be any open set.
    \item The auras can have different sizes and shapes for different points.
    \item The topological structure of $\tau$ constrains the auras, providing a natural geometric framework.
\end{enumerate}
When $\aura$ is symmetric and transitive, the model reduces to a partition-based approximation similar to Pawlak's.
\end{remark}

\begin{definition}\label{def:rough-measures}
Let $(X, \tau, \aura)$ be an $\aura$-space with $X$ finite and $A \subseteq X$ with $\Apr(A) \neq \emptyset$. The \textbf{aura accuracy measure} and \textbf{aura roughness measure} of $A$ are:
\begin{equation}
    \rho_\aura(A) = \frac{|\apr(A)|}{|\Apr(A)|}, \quad \sigma_\aura(A) = 1 - \rho_\aura(A).
\end{equation}
$A$ is called \textbf{$\aura$-definable} (or \textbf{$\aura$-exact}) if $\apr(A) = \Apr(A)$ (equivalently, $\rho_\aura(A) = 1$).
\end{definition}

\subsection{Application: Medical Decision Making}

We illustrate the applicability of aura-based rough sets with a medical decision-making example.

\begin{example}\label{ex:medical}
Consider a set of patients $U = \{p_1, p_2, p_3, p_4, p_5, p_6\}$ evaluated by a medical system with symptoms and diagnostic data. Suppose a topology $\tau$ is induced on $U$ by the diagnostic similarity, and an expert assigns each patient an aura representing the group of patients considered ``diagnostically similar'' based on a specific clinical criterion:

\begin{center}
\begin{tabular}{c|c|c|c}
\hline
Patient & Symptoms & Diagnosis & $\aura(p_i)$ \\
\hline
$p_1$ & Fever, Cough & Flu & $\{p_1, p_2\}$ \\
$p_2$ & Fever, Fatigue & Flu & $\{p_1, p_2, p_3\}$ \\
$p_3$ & Fatigue, Headache & Uncertain & $\{p_2, p_3\}$ \\
$p_4$ & Cough, Chest pain & Pneumonia & $\{p_4, p_5\}$ \\
$p_5$ & Chest pain, Fever & Pneumonia & $\{p_4, p_5, p_6\}$ \\
$p_6$ & Fever, Fatigue & Uncertain & $\{p_5, p_6\}$ \\
\hline
\end{tabular}
\end{center}

Let the topology $\tau$ on $U$ be generated by the subbasis $\{\aura(p_i) : 1 \leq i \leq 6\}$. Consider the target set $A = \{p_1, p_2, p_4, p_5\}$ (patients diagnosed with a definite disease: Flu or Pneumonia).

\textbf{Aura lower approximation:}
\begin{align*}
\apr(A) &= \{p_i \in A : \aura(p_i) \subseteq A\} \\
&= \{p_1, p_4\},
\end{align*}
since $\aura(p_1) = \{p_1, p_2\} \subseteq A$, $\aura(p_2) = \{p_1, p_2, p_3\} \not\subseteq A$ (because $p_3 \notin A$), $\aura(p_4) = \{p_4, p_5\} \subseteq A$, and $\aura(p_5) = \{p_4, p_5, p_6\} \not\subseteq A$ (because $p_6 \notin A$).

\textbf{Aura upper approximation:}
\begin{align*}
\Apr(A) &= \{p_i \in U : \aura(p_i) \cap A \neq \emptyset\} \\
&= \{p_1, p_2, p_3, p_4, p_5, p_6\} = U,
\end{align*}
since every patient's aura intersects $A$.

\textbf{Boundary region:}
\[
\operatorname{bnd}_\aura(A) = U \setminus \{p_1, p_4\} = \{p_2, p_3, p_5, p_6\}.
\]

\textbf{Accuracy:} $\rho_\aura(A) = \frac{2}{6} = 0.333$.

\textbf{Interpretation:} Patients $p_1$ and $p_4$ are \emph{certainly} in the diagnosed group (all their similar patients are also diagnosed). Patients $p_2, p_3, p_5, p_6$ are in the boundary---their diagnostic status is uncertain because their aura overlaps with undiagnosed patients. The low accuracy reflects the ambiguity in the diagnostic data.

Now, if we refine the aura by using a more specific clinical criterion (e.g., $\aura'(p_2) = \{p_1, p_2\}$ and $\aura'(p_5) = \{p_4, p_5\}$, keeping others the same), then $\underline{\operatorname{apr}}_{\aura'}(A) = \{p_1, p_2, p_4, p_5\} = A$, giving $\rho_{\aura'}(A) = 1$. This demonstrates that refining the aura function improves diagnostic certainty.
\end{example}

\begin{theorem}\label{thm:refine}
Let $(X, \tau, \aura_1)$ and $(X, \tau, \aura_2)$ be two $\aura$-spaces on the same topological space. If $\aura_2(x) \subseteq \aura_1(x)$ for every $x \in X$ (i.e., $\aura_2$ is a \emph{refinement} of $\aura_1$), then for every $A \subseteq X$:
\begin{enumerate}[label=(\alph*)]
    \item $\underline{\operatorname{apr}}_{\aura_1}(A) \subseteq \underline{\operatorname{apr}}_{\aura_2}(A)$ (lower approximation increases);
    \item $\overline{\operatorname{apr}}_{\aura_2}(A) \subseteq \overline{\operatorname{apr}}_{\aura_1}(A)$ (upper approximation decreases);
    \item $|\operatorname{bnd}_{\aura_2}(A)| \leq |\operatorname{bnd}_{\aura_1}(A)|$ when $X$ is finite (boundary shrinks).
\end{enumerate}
\end{theorem}

\begin{proof}
\begin{enumerate}[label=(\alph*)]
    \item If $x \in \underline{\operatorname{apr}}_{\aura_1}(A)$, then $\aura_1(x) \subseteq A$. Since $\aura_2(x) \subseteq \aura_1(x) \subseteq A$, $x \in \underline{\operatorname{apr}}_{\aura_2}(A)$.
    \item If $x \in \overline{\operatorname{apr}}_{\aura_2}(A)$, then $\aura_2(x) \cap A \neq \emptyset$. Since $\aura_2(x) \subseteq \aura_1(x)$, $\aura_1(x) \cap A \neq \emptyset$, so $x \in \overline{\operatorname{apr}}_{\aura_1}(A)$.
    \item Follows from (a) and (b): $\operatorname{bnd}_{\aura_2}(A) = \overline{\operatorname{apr}}_{\aura_2}(A) \setminus \underline{\operatorname{apr}}_{\aura_2}(A) \subseteq \overline{\operatorname{apr}}_{\aura_1}(A) \setminus \underline{\operatorname{apr}}_{\aura_1}(A) = \operatorname{bnd}_{\aura_1}(A)$. \qedhere
\end{enumerate}
\end{proof}

\subsection{Wireless Sensor Network Coverage}\label{sec:sensor}

Consider a region $X \subseteq \R^2$ equipped with the usual topology $\tau_u$, and suppose that a finite set of sensors $S = \{s_1, \ldots, s_n\} \subset X$ is deployed. Each sensor $s_i$ has a detection range $r_i > 0$, so it monitors the open disk $B(s_i, r_i) = \{y \in X : d(s_i, y) < r_i\}$. For an arbitrary point $x \in X$ we define $\aura(x)$ as the detection zone of the sensor that is ``responsible'' for $x$: if $x = s_i$ we set $\aura(s_i) = B(s_i, r_i)$; if $x$ is not a sensor we set $\aura(x) = B(x, \delta)$ for some small fixed $\delta > 0$ (say, a local measurement radius). Then $\aura: X \to \tau_u$ is a scope function and $(X, \tau_u, \aura)$ is an aura space. The aura topology $\taua$ encodes the effective coverage geometry.

\begin{example}[Coverage analysis]\label{ex:sensor}
Let $X = \R^2$, and place three sensors with positions and ranges
\[
s_1 = (0,0),\; r_1 = 3;\qquad s_2 = (4,0),\; r_2 = 2;\qquad s_3 = (2,3),\; r_3 = 2.
\]
Assign each point its nearest sensor's disk as its aura (ties broken arbitrarily). Consider the target region $A = [1,3] \times [0,2]$.

\textbf{Lower approximation.}
$\apr(A) = \{x \in A : \aura(x) \subseteq A\}$ consists of points whose entire sensor disk lies inside $A$. Since all sensor radii are at least $2$, no sensor disk fits entirely inside the $2 \times 2$ box, so $\apr(A) = \emptyset$.

\textbf{Upper approximation.}
$\Apr(A) = \{x \in X : \aura(x) \cap A \neq \emptyset\}$ consists of points whose sensor disk touches $A$. All three sensor disks overlap $A$, so $\Apr(A)$ contains a large neighbourhood of $A$.

\textbf{Interpretation.}
$\apr(A) = \emptyset$ tells us that no sensor can guarantee full surveillance of $A$ on its own. $\Apr(A) \supsetneq A$ tells us that the sensor network detects any event happening in $A$, but may also generate false positives from outside $A$. The boundary $\operatorname{bnd}_\aura(A)$ quantifies the zone of surveillance uncertainty.
\end{example}

The iterative aura-closure has a direct meaning here. If a sensor at $x$ detects an event, every point in $\aura(x)$ is alerted. If sensors relay information, the second closure $\cla^2(A)$ gives the set of points alerted after one relay step, $\cla^3(A)$ after two relay steps, and so on. The stabilisation $\cla^\infty(A)$ is the set of all points that can eventually be reached by relay from $A$.

\begin{proposition}\label{prop:full-coverage}
Let $(X, \tau, \aura)$ be an aura space with $X$ finite. The sensor network provides \textbf{full coverage} of a target set $A$ (in the sense that $\apr(A) = A$) if and only if $A$ is $\aura$-open.
\end{proposition}

\begin{proof}
$\apr(A) = \inta(A) = A$ if and only if $A$ is $\aura$-open, by definition.
\end{proof}

\begin{remark}\label{rem:coverage-refinement}
By Theorem~\ref{thm:refine}, deploying more sensors (which effectively refines $\aura$) increases the lower approximation and decreases the boundary. This gives a topological criterion for sensor placement optimisation: add sensors until the target set becomes $\aura$-open.
\end{remark}

\subsection{Modelling Epidemic Spread}\label{sec:epidemic}

Let $X$ be a population (finite or countable) and let $\tau$ be a topology on $X$ generated by social or geographical proximity. For each individual $x \in X$ define $\aura(x) \in \tau$ to be the set of individuals whom $x$ can directly infect---the \emph{transmission neighbourhood} of $x$. The axiom $x \in \aura(x)$ is natural: an infected person is in contact with himself. The triple $(X, \tau, \aura)$ models the epidemic environment.

\begin{example}[Spread dynamics on a small population]\label{ex:epidemic}
Take $X = \{a, b, c, d, e, f, g\}$ with $\tau = \powerset(X)$ (so that any subset can serve as an aura), and define the transmission neighbourhoods
\[
\aura(a) = \{a, b\}, \quad \aura(b) = \{b, c, d\}, \quad \aura(c) = \{c\}, \quad \aura(d) = \{d, e\},
\]
\[
\aura(e) = \{e, f\}, \quad \aura(f) = \{f\}, \quad \aura(g) = \{g\}.
\]
Suppose the initial set of infected individuals is $A_0 = \{a\}$.

\textbf{Step 1.} $\cla(A_0) = \{x : \aura(x) \cap \{a\} \neq \emptyset\} = \{a\}$, since only $\aura(a)$ contains $a$. So after one step, only $a$ is infected. But $a$ can transmit to $\aura(a) = \{a, b\}$. To model forward spread, we use
\[
A_1 = \bigcup_{x \in A_0} \aura(x) = \aura(a) = \{a, b\}.
\]

\textbf{Step 2.} $A_2 = \bigcup_{x \in A_1} \aura(x) = \aura(a) \cup \aura(b) = \{a, b\} \cup \{b, c, d\} = \{a, b, c, d\}$.

\textbf{Step 3.} $A_3 = A_2 \cup \aura(c) \cup \aura(d) = \{a, b, c, d\} \cup \{c\} \cup \{d, e\} = \{a, b, c, d, e\}$.

\textbf{Step 4.} $A_4 = A_3 \cup \aura(e) = \{a, b, c, d, e, f\}$.

\textbf{Step 5.} $A_5 = A_4 \cup \aura(f) = A_4$. Stabilisation reached.

The set $A_\infty = \{a, b, c, d, e, f\}$ is the ``epidemic reach'' of $a$. Individual $g$ is never infected because no transmission chain connects $a$ to $g$. The number of steps to stabilisation ($n = 4$) measures the depth of the infection chain.
\end{example}

We formalise the forward spread operator as follows.

\begin{definition}\label{def:spread}
Let $(X, \tau, \aura)$ be an aura space. The \textbf{spread operator} is defined by
\[
S_\aura(A) = \bigcup_{x \in A} \aura(x)
\]
for $A \subseteq X$. The \textbf{$n$-step spread} is $S_\aura^n(A)$ (iterate $n$ times), and the \textbf{total spread} is $S_\aura^\infty(A) = \bigcup_{n \geq 0} S_\aura^n(A)$.
\end{definition}

\begin{proposition}\label{prop:spread}
Let $(X, \tau, \aura)$ be an aura space and $A \subseteq X$. Then:
\begin{enumerate}[label=(\alph*)]
    \item $A \subseteq S_\aura(A) \subseteq \cla(S_\aura(A))$.
    \item If $\aura$ is transitive, then $S_\aura(A)$ is $\aura$-open for every $A$.
    \item $S_\aura^\infty(A) = S_\aura^\infty(\{x\})$ for some $x \in A$ if and only if $A$ is contained in a single $\aura$-connected component (when $\aura$ is transitive).
\end{enumerate}
\end{proposition}

\begin{proof}
(a) For any $x \in A$, $x \in \aura(x) \subseteq S_\aura(A)$, so $A \subseteq S_\aura(A)$. For any $y \in S_\aura(A)$, $y \in \aura(x_0)$ for some $x_0 \in A$, so $\aura(y) \cap S_\aura(A) \supseteq \{y\} \neq \emptyset$ and $y \in \cla(S_\aura(A))$.

(b) If $\aura$ is transitive and $y \in S_\aura(A)$, then $y \in \aura(x_0)$ for some $x_0 \in A$. For any $z \in \aura(y)$, transitivity gives $\aura(z) \subseteq \aura(y) \subseteq \aura(x_0)$, so $z \in S_\aura(A)$. Hence $\aura(y) \subseteq S_\aura(A)$ and $S_\aura(A)$ is $\aura$-open.

(c) In a transitive aura space, $S_\aura^\infty(\{x\})$ is the $\aura$-connected component of $x$ (the maximal $\aura$-connected set containing $x$), since $S_\aura^\infty(\{x\})$ is $\aura$-open by (b) and its complement is also $\aura$-open. The claim follows.
\end{proof}

\begin{remark}\label{rem:epidemic-control}
The epidemic model suggests natural control strategies in terms of the aura function:
\begin{enumerate}[label=(\roman*)]
    \item \textbf{Quarantine} corresponds to replacing $\aura(x)$ by $\{x\}$ for infected individuals, which cuts the transmission chain.
    \item \textbf{Social distancing} corresponds to refining $\aura$ globally---replacing each $\aura(x)$ by a smaller open set $\aura'(x) \subseteq \aura(x)$. By Theorem~\ref{thm:refine}, this shrinks the upper approximation and the boundary.
    \item \textbf{Herd immunity} can be interpreted as making the set of immune individuals $\aura$-open: once $I \subseteq X$ is $\aura$-open, the spread operator satisfies $S_\aura(X \setminus I) \cap I = \emptyset$ for transitive $\aura$.
\end{enumerate}
\end{remark}

\section{Conclusion}

We have introduced the aura topological space $(X, \tau, \aura)$, a structure obtained by attaching to each point $x$ a fixed open neighborhood $\aura(x)$. This simple axiom---a single function $\aura: X \to \tau$ with $x \in \aura(x)$---turned out to be surprisingly productive.

On the operator side, the aura-closure $\cla$ is an additive \v{C}ech closure operator whose failure of idempotency leads naturally to the transfinite iteration $\cla^\infty$ and a Kuratowski closure. The two resulting topologies sit in the chain $\taua^\infty \subseteq \taua \subseteq \tau$, and the gap between them collapses precisely when $\aura$ is transitive.

On the open-set side, combining $\cla$ with the standard interior produces five generalized open-set classes ($\aura$-semi-open, $\aura$-pre-open, $\aura$-$\alpha$-open, $\aura$-$\beta$-open, $\aura$-$b$-open) whose hierarchy we determined completely, separating all non-coinciding classes by explicit examples on finite sets and on $\R$. The corresponding continuity notions admit decomposition theorems that extend the classical results of Levine and Mashhour. The separation axioms $\aura$-$T_0$, $\aura$-$T_1$, $\aura$-$T_2$ depend on the choice of the scope function, so that the same underlying space can exhibit different separation behaviour for different aura functions.

On the applied side, the aura framework lends itself to at least three types of problems: rough set approximation without equivalence relations (where refining the aura function improves accuracy---Theorem~\ref{thm:refine}), wireless sensor coverage analysis (where full coverage of a target set is equivalent to the set being $\aura$-open---Proposition~\ref{prop:full-coverage}), and epidemic spread modelling (where the spread operator $S_\aura$ tracks infection chains and standard interventions such as quarantine and social distancing correspond to modifying the scope function---Remark~\ref{rem:epidemic-control}). These examples suggest that the aura space may serve as a unifying framework for problems involving ``local range'' or ``local scope'' data.

Several questions are left for future work. The covering properties (compactness, Lindel\"of) and connectivity theory for aura spaces deserve a systematic treatment; the same holds for product and subspace constructions. Combining the aura function with an ideal opens the way to a hybrid local function and a topology that interpolates between $\taua$ and $\tau^*$. On the applied side, more detailed models for sensor deployment optimisation, epidemic threshold analysis, and digital image segmentation are worth exploring.

\section*{Conflict of Interest}

The author declares no conflict of interest.

\section*{Data Availability}

No data was used for the research described in the article.

\end{document}